\documentclass[11pt,leqno,twoside,letterpaper]{amsart}

\renewcommand{\baselinestretch}{1.1}
\usepackage{amsaddr}
\usepackage{mathptmx}
\usepackage[english]{babel}
\usepackage[utf8]{inputenc}
\usepackage[T1]{fontenc}
\usepackage{csquotes}
\usepackage{enumitem}
\usepackage{etoolbox}

\DeclareMathAlphabet{\mathcal}{OMS}{cmsy}{m}{n}

\usepackage{amssymb,amsmath,mathrsfs,textcomp}

\usepackage{graphicx,color,float}
\usepackage{tikz,tikz-cd}
\tikzcdset{every diagram/.code={\tikzcdset{row sep={1.4cm,between origins},column sep={1.5cm,between origins}}}}
\usepackage[labelfont=bf]{caption}
\usepackage{color}
\definecolor{aCol}{RGB}{0,0,166}
\usepackage[colorlinks=true,urlcolor=aCol,linkcolor=aCol,citecolor=aCol]{hyperref}

\makeatletter
    \def\@bbhspace{.02em}
    \def\RR{\hspace{\@bbhspace}{\bf R}\hspace{\@bbhspace}}
    \def\bS{\hspace{\@bbhspace}{\bf S}\hspace{\@bbhspace}}
    \def\CC{\hspace{\@bbhspace}{\bf C}\hspace{\@bbhspace}}
    \def\ZZ{\hspace{\@bbhspace}{\bf Z}\hspace{\@bbhspace}}
    \def\NN{\hspace{\@bbhspace}{\bf N}\hspace{\@bbhspace}}
    \def\PP{\hspace{\@bbhspace}{\bf P}\hspace{\@bbhspace}}
    \def\BB{\hspace{\@bbhspace}{\bf B}\hspace{\@bbhspace}}
    \def\CP{\hspace{\@bbhspace}{\bf CP}\hspace{\@bbhspace}}
    \def\RP{\hspace{\@bbhspace}{\bf RP}\hspace{\@bbhspace}}
\makeatother
\def\RRb{\overline{\RR}}
\def\CPb{\overline{\CP}}
\def\RPb{\overline{\RP}}
\def\conj{\mathrm{conj}}
\def\cR{\mathcal{R}}
\def\cQ{\mathcal{Q}}
\def\cA{\mathcal{A}}
\def\cX{\mathcal{X}}
\def\Qb{\overline{Q}}
\def\tcap{\pitchfork}
\def\Aut{\mathrm{Aut}}
\def\fR{\mathfrak{R}}
\def\fQ{\mathfrak{Q}}
\def\Xb{\overline{X}}
\def\fRb{\overline{\mathfrak{R}}}
\def\fQb{\overline{\mathfrak{Q}}}
\def\msC{\mathscr{C}}
\def\msD{\mathscr{D}}

\let\oldell\ell
\def\ell{\oldell\hspace{.05em}}

\usepackage{amsthm}

\makeatletter
\def\thmhead@plain#1#2#3{%
  \thmname{#1}\thmnumber{\@ifnotempty{#1}{ }\@upn{#2}}%
  \thmnote{ {\the\thm@notefont#3}}}
\let\thmhead\thmhead@plain
\makeatother

\newtheorem{theorem}{Theorem}[section]
\newtheorem{proposition}[theorem]{Proposition}
\newtheorem{corollary}[theorem]{Corollary}
\newtheorem{lemma}[theorem]{Lemma}
\newtheorem{conjecture}[theorem]{Conjecture}

\theoremstyle{definition}
\newtheorem{definition}[theorem]{Definition}

\addto\captionsenglish{}

\patchcmd{\section}{\scshape}{\bfseries}{}{}
\patchcmd{\subsection}{\bfseries}{\bfseries\itshape}{}{}
\makeatletter
    \renewcommand{\@secnumfont}{\bfseries}
\makeatother

\newcommand\blfootnote[1]{%
  \begingroup
  \renewcommand\thefootnote{}\footnote{#1}%
  \addtocounter{footnote}{-1}%
  \endgroup
}

\makeatletter
\let\std@footnotetext\@footnotetext
\usepackage{setspace}
\let\@footnotetext\std@footnotetext
\patchcmd{\@footnotetext}
  {\normalfont}
  {\def\baselinestretch{\setspace@singlespace}\reset@font\normalfont}
  {}{}
\makeatother

\def\cref[#1]#2{\hyperref[#2]{#1~\ref*{#2}}}
\def\ceqref#1{{\color{aCol}(\ref{#1})}}

\makeatletter
    \newtheorem*{rep@theorem}{\rep@title}
    \newcommand{\newreptheorem}[2]{%
    \newenvironment{rep#1}[1]{\def\rep@title{\!\! \cref[Theorem]{##1}}\begin{rep@theorem}}{\end{rep@theorem}}}
\makeatother
\newreptheorem{theorem}{Theorem}
\newreptheorem{definition}{Definition}

\usepackage[natbibapa]{apacite}

\let\oldciteyear\citeyear
\renewcommand{\citeyear}[2][]{{\color{aCol}[}\oldciteyear[#1]{#2}{\color{aCol}]}}

\let\oldciteauthor\citeauthor
\renewcommand{\citeauthor}[2][]{{\color{aCol}[}\oldciteauthor[#1]{#2}{\color{aCol}]}}

\renewcommand{\cite}[2][]{%
    \ifx&#1&%
        {\color{aCol}[}\oldciteauthor{#2}\ \oldciteyear{#2}{\color{aCol}]}%
    \else%
        {\color{aCol}[}\oldciteauthor{#2}\ \oldciteyear{#2},\ #1{\color{aCol}]}%
    \fi%
}

\def\citeLRS{{\color{aCol}[}\hyperlink{cite.LRS15}{Levine et al. 2015}, Theorem 10.1{\color{aCol}]}}

\begin{document}

\title[A Viro--Zvonilov-type inequality]{A Viro--Zvonilov-type inequality\\for Q-flexible curves of odd degree\vspace{-1.6\baselineskip}}

\author[Anthony Saint-Criq]{\textsc{Anthony Saint-Criq}}
\let\origmaketitle\maketitle
\def\maketitle{
  \begingroup
  \let\MakeUppercase\relax
  \origmaketitle
  \endgroup
}

%\vspace*{-12pt}
%\begin{flushright}
%    {\tiny\color{aCol}\href{https://msp.org/pjm/}{PACIFIC JOURNAL OF MATHEMATICS}\\
%    \href{https://doi.org/10.2140/pjm.2024.-}{Vol. , No. , 2024}\\\vspace{-.4\baselineskip}
%    \href{https://doi.org/10.2140/pjm.2024..101}{\texttt{http://doi.org/10.2140/pjm.2024..101}}}
%\end{flushright}
%\vspace{.7\baselineskip}

\maketitle

\vspace{3.5\baselineskip}
\begin{center}\begin{minipage}{.89\textwidth}
    \textbf{\normalsize
    We define an analogue of the Arnold surface for odd degree flexible curves, and we use it to double branch cover $Q$-flexible embeddings, where $Q$-flexible is a condition to be added to the classical notion of a flexible curve. This allows us to obtain a Viro--Zvonilov-type inequality: an upper bound on the number of non-empty ovals of a curve of odd degree. We investigate our method for flexible curves in a quadric to derive a similar bound in two cases. We also digress around a possible definition of non-orientable flexible curves, for which our method still works and a similar inequality holds.}
\end{minipage}\end{center}
\vspace{3.2\baselineskip}
\blfootnote{\hspace{-1.3em}\textit{MSC2020:} primary 14P25, 57M12; secondary 14H45, 57S25.

\hspace{-.2em}nonorientable surfaces.}
\medskip

Let $F\subset\CP^2$ be a flexible curve of odd degree $m$. We denote as $\ell^\pm$ and $\ell^0$ the number of ovals of the curve $\RR F\subset\RP^2$ that bound from the outside a component of $\RP^2\smallsetminus\RR F$ which has positive, negative or zero Euler characteristic, respectively. In particular, $\ell^+$ is the number of empty ovals, and $\ell^0+\ell^-$ is that of nonempty\linebreak
ones. O. Viro and V. Zvonilov \citeyear{VZ92} proved the following upper bound for the number of non-empty ovals: $$\ell^0+\ell^-\leqslant\frac{(m-3)^2}{4}+\frac{m^2-h(m)^2}{4h(m)^2},$$
with $h(m)$ denoting the biggest prime power that divides $m$. Their proof relied on taking a branched cover of $\CP^2$, ramified over the surface $F$. Usually, it is a good choice to take doubly sheeted branched covers, but this is not possible in this setting where $m$ is odd. Odd degree curves are a different story compared to even degree ones, one reason being the nonexistence of the Arnold surface in $\bS^4$ ($\RR F$ is not null-homologous in $H_1(\RP^2;\ZZ/2)$, and neither is $F$ in $H_2(\CP^2;\ZZ/2)$). In the present paper, we give a definition of an analogue of the Arnold surface in $\CPb^2$ for odd degree curves. This means that, under a certain condition of being $Q$-flexible (up to taking another conic $Q$ with empty real part and pseudo-holomorphic, this is always satisfied by pseudo-holomorphic curves), we are allowed to take the double branched cover of $\CPb^2$ ramified over a perturbation of this Arnold surface. This condition is also always satisfied by algebraic curves. We will show the following result, by methods analogous to Viro and Zvonilov.

\begin{reptheorem}{thm:main}
    Let $F$ be a $Q$-flexible curve of odd degree $m$. Then $$\ell^0+\ell^-\leqslant\frac{(m-1)^2}{4}.$$
    If equality holds, then the curve is type I.
\end{reptheorem}

It is worth mentioning that this is not quite Zvonilov's bound $(m-1)(m-3)/4$ \citeyear{Zvo79}, which works for any flexible curve that intersects a real line generically (this condition being the degree one analogue of our $Q$-flexibility), and in particular for any pseudo-holomorphic curve. However, it appears that $Q$-flexibility and this condition by Zvonilov are independent for general flexible curves.

In \cref[Section]{sec:prelim}, we discuss some constructions in $\CP^2$ and $\CPb^2$ seen as 2-fold branched covers of the standard 4-sphere. In \cref[Section]{sec:arnold}, we construct the Arnold surface for odd degree curves, and we describe the behavior of the real part of the curve under this construction. In \cref[Section]{sec:proof}, we prove the inequality. In \cref[Section]{sec:quadrics},\linebreak
we review our method for curves in a quadric to produce a result which, to our knowledge,\hspace{-.05em} is new even for algebraic curves.\hspace{-.05em} In \cref[Section]{sec:further},\hspace{-.05em} we compare our \mbox{inequality}\linebreak{}
to Viro and Zvonilov's, and we investigate the possible notion of nonorientable\linebreak
flexible curves, for which our method still applies to derive a similar bound.

\section{Preliminaries}\label{sec:prelim}\noindent
    Throughout this paper, all surfaces will be assumed to be connected, and all embed\-dings are smooth.
    
    The complex conjugation $\conj:\CP^2\to\CP^2$ is defined in homogeneous coordinates by $\conj([z_0:z_1:z_2])=[\overline{z_0}:\overline{z_1}:\overline{z_2}]$, and $\RP^2\subset\CP^2$ is the fix-point set $\mathrm{Fix}(\conj)$. Here, $Q$ will always denote a generic real conic with empty real part (for instance, the Fermat conic given by the equation $z_0^2+z_1^2+z_2^2=0$). In particular, it is a smoothly embedded $2$-sphere $Q\subset\CP^2$ which represents the homology class $[Q]=2[\CP^1]$ in $H_2(\CP^2;\ZZ)\cong\ZZ$, the choice of a generator being the homology class of any complex line.
    \subsection*{\texorpdfstring{Flexible and $Q$-flexible curves}{Flexible and Q-flexible curves}}

A \textit{real plane algebraic curve} is a real \textit{nonsingular} homogenous polynomial $X\in\RR[x_0:x_1:x_2]$. By the \textit{real part} of the curve, we mean the set $$\RR X=\left\{[x_0:x_1:x_2]\in\RP^2\mid X(x_0,x_1,x_2)=0\right\},$$
and by the \textit{complexification} of the curve, we mean $$\CC X=\left\{[z_0:z_1:z_2]\in\CP^2\mid X(z_0,z_1,z_2)=0\right\}.$$
Evidently, $\CC X$ is invariant under complex conjugation. If $m=\deg(X)\geqslant1$, then\linebreak
we see that $[\CC X]=m[\CP^1]\in H_2(\CP^2;\ZZ)$, and that $\CC X$ is a surface of genus\linebreak$g=(m-1)(m-2)/2$. Also, the tangent space of the complex curve is related~to that of the real curve in the following sense: $$\text{for all }x\in\RR X,\quad T_x\CC X=T_x\RR X\oplus{\bf i}\cdot T_x\RR X.$$
We define a \textit{flexible} curve, in the sense of Viro \citeyear{Vir84}, as follows:

\begin{definition}\label{def:flexible}
    Let $F$ be a closed oriented surface embedded in $\CP^2$. The surface~$F$ is called a degree $m\geqslant1$ \textit{flexible curve} if
    \begin{enumerate}[label=(\roman*),leftmargin=21pt,itemsep=3pt]
        \item $\conj(F)=F$;
        \item $\chi(F)=-m^2+3m$;
        \item $[F]=m[\CP^1]\in H_2(\CP^2;\ZZ)$;
        \item for all $x\in\RR F,\;T_xF=T_x\RR F\oplus{\bf i}\cdot T_x\RR F$, where $\RR F=F\cap\RP^2$.
    \end{enumerate}
\end{definition}

The following classical results are known for flexible curves (see \cite{Rok78}):
\begin{enumerate}[label=(\roman*),leftmargin=21pt,itemsep=3pt]
    \item If $b_0$ denotes the 0-th Betti number, then ${b_0(\RR F)\leqslant g\!+\!1\!=\!(m\!-\!1)(m\!-\!2)/2\!+\!1}$. Curves with $b_0(\RR F)=g+1$ are called \textit{$M$-curves}. On the other hand, curves with $b_0(\RR F)=0$ if $m$ is even and $b_0(\RR F)=1$ if $m\geqslant 3$ is odd are called\linebreak
    \textit{minimal} curves (note that there are no minimal curves in degree one).
    \item If $m$ is even, then each component of $\RR F$ is contractible in $\RP^2$, and if $m$ is odd, all but one components of $\RR F$ are. Contractible components of $\RR F$ are called \textit{ovals}. An odd degree curve can never have $\RR F=\varnothing$ (it always has the non-contractible component).
    \item A flexible curve $F$ is said to be a \textit{type I} (\textit{resp.} \textit{type II}) curve if $F\smallsetminus\RR F$ has two connected components (\textit{resp.} is connected). An $M$-curve is always type I,\linebreak
    and a minimal curve is always type II.
\end{enumerate}

What makes flexible curves so different from algebraic curves is the lack of\linebreak
rigidity, mainly seen with the Bézout theorem, which, in particular, implies that a degree $m$ algebraic curve generically intersects $Q$ transversely in exactly $2m$ points.

\begin{definition}
    A flexible curve $F$ of degree $m\geqslant1$ is called \textit{$Q$-flexible} if $F\tcap Q$ consists of $2m$ points, necessarily swapped pairwise by complex conjugation.
\end{definition}

%Is it always possible to isotope a flexible curve into a totally flexible one, sharing the same real scheme? That is: can we find an isotopy $F'$ of $F$ in $\CP^2$, fixing $F\cap M_\varepsilon^+$, such that $F'\tcap Q=2m\ast$ and $\conj(F')=F'$? Note: a flexible curve needs not have \textit{at most} $2m$ intersection points, it can have more...
    \subsection*{Two double branched covers}

There is a well-known diffeomorphism between\linebreak$\CP^2/\conj$ and $\bS^4$ (see \cite{Kui74}). We denote the associated (cyclic) 2-fold branched cover as $p:(\CP^2,\RP^2)\to(\bS^4,\cR)$, with $\cR=p(\RP^2)$ an embedded $\RP^2$ in $\bS^4$. We also let $\cQ=p(Q)$, the image of the preferred conic under that branched cover. From the fact that $Q$ does not intersect the branch locus $\RP^2$, we see that the restriction $p:Q\to\cQ$ is an unbranched 2-fold cover, with the conjugation being an orientation-preserving involution generating the group of deck transformations. As such, we see that $\cQ$ is also an embedded $\RP^2$ in $\bS^4$.

Given a closed embedded surface $F^2$ in a (closed oriented) 4-manifold $X^4$, we denote as $e(X,F)$ the normal Euler number of the embedding $F\subset X$; that is,\linebreak
the Euler class of the normal bundle $\nu F$. It is also equal to the self-intersection number $F\cdot F$, which is defined by counting signed intersection points between $F$ and a small perturbation $F'$ of $F$ in the normal direction. If $F$ is oriented, then it corresponds to the intersection form of $X$ evaluated on $[F]\in H_2(X;\ZZ)$.

\begin{proposition}
    We have the following normal Euler numbers: $$e(\CP^2,\RP^2)=-1,\quad e(\CP^2,Q)=+4,\quad e(\bS^4,\cR)=-2\quad\text{and}\quad e(\bS^4,\cQ)=+2.$$
\end{proposition}

\begin{proof}
    The conic $Q$ is oriented and $[Q]=2[\CP^1]\in H_2(\CP^2;\ZZ)$, so we obtain\linebreak
    $e(\CP^2,Q)=+4$. Next, because $\RP^2$ is Lagrangian in $\CP^2$, we have that the normal bundle $\nu\RP^2$ and the tangent bundle $T\RP^2$ are anti-isomorphic, and thus, for the Euler class, $e(\nu\RP^2)=-e(T\RP^2)=-\chi(\RP^2)=-1$. Finally, the computations of $e(\bS^4,\cR)$ and $e(\bS^4,\cQ)$ come from the next lemma.
\end{proof}

\begin{lemma}\label{lem:branched-cover-euler-number}
    Given a 2-fold branched cover $f:(Y^4,\tilde{B}^2)\to(X^4,B^2)$, and given\linebreak
    $F$ an embedded closed surface in $X$, we denote as $\tilde{F}$ the lift $p^{-1}(F)$.
    \begin{itemize}[leftmargin=21pt,itemsep=3pt]
        \item[\textup{(i)}] If $F\tcap B$, possibly with $F\cap B=\varnothing$, then $e(Y,\tilde{F})=2e(X,F)$.
        \item[\textup{(ii)}] If $F\subset B$, then $e(Y,\tilde{F})=\frac{1}{2}e(X,F)$.
    \end{itemize}
\end{lemma}

\begin{proof}
    One has to inspect what happens in each case individually. In the first, note that the lift of a perturbation is a perturbation of the lift, and one can ensure that~the self-intersection points occur away from the ramification locus $B$. As such, each of these points lifts to two intersections, and the orientations agree because $f$ is orientation-preserving.
    
    The second case can be deduced from the first. Let $\tilde{F}'$ be a small transverse perturbation of $\tilde{F}$. Letting $\tau:Y\to Y$ denote the involution that spans $\Aut(f)$, and letting $F'=f(\tilde{F}')$, we see that $F'$ is a perturbation of $F$ and $\tilde{F}'\cup\tau(\tilde{F}')$ is the lift\linebreak
    of $F'$. By the first case, we obtain $e(\tau(\tilde{F}'))=e(\tilde{F}')=2e(F')=2e(F)$. Moreover, we have $2e(\tilde{F})=e(\tilde{F}'\cup\tau(\tilde{F}'))=e(\tilde{F}')+e(\tau(\tilde{F}'))=4e(F)$.
\end{proof}

We now wish to consider the 2-fold branched cover of $\bS^4$, ramified over $\cQ$ this time. It is possible to make some computations to find an orientation-reversing\linebreak
involution of $\bS^4$ which swaps $\cR$ and $\cQ$. Alternatively, taking any orientation-reversing free involution of $\bS^4$, this maps $\cR$ to a projective plane with normal Euler number $+2$, and this is always isotopic to $\cQ$ in $\bS^4$. Tracking this isotopy produces the involution needed. As such, we see that the smooth $4$-manifold obtained as the double branched cover of $\bS^4$ ramified along $\cQ$ is diffeomorphic\footnote{\setstretch{.5}In fact, it is sufficient to obtain that the double branch cover of $\bS^4$ ramified over $\cQ$ is a homol-\linebreak
ogy $\CPb^2$, as will be the case for curves on quadrics in a later section.} to $\CPb^2$. We let $\tilde{p}:\CPb^2\to\bS^4$ denote that double branched cover.

Define $\Qb=\tilde{p}^{-1}(\cR)$. We see that $\RPb^2$ and $\Qb$ are respectively embeddings of~$\RP^2$ and $\bS^2$ in $\CPb^2$. Using \cref[Lemma]{lem:branched-cover-euler-number} again, we can compute the normal Euler numbers.

\begin{proposition}
    We have $e(\CPb^2,\RPb^2)=+1$ and $e(\CPb^2,\Qb)=-4$.
\end{proposition}

In Figure \ref{fig:diagram} we depict a summary of the different maps in play.

\begin{figure}[t]
    \centering
    \begin{tikzcd}
    	{\RP^2} & {\CP^2} & Q \\
    	\cR & {\bS^4} & \cQ \\
    	\Qb & {\CPb^2} & {\RPb^2}
    	\arrow["{\tilde{p}}"', from=3-2, to=2-2]
    	\arrow["p"', from=1-2, to=2-2]
    	\arrow["{||}"{description}, draw=none, from=1-1, to=2-1]
    	\arrow["{||}"{description}, draw=none, from=3-3, to=2-3]
    	\arrow[shorten <=1pt, shorten >=1pt, two heads, from=1-3, to=2-3]
    	\arrow[shorten <=2pt, shorten >=2pt, two heads, from=3-1, to=2-1]
    	\arrow["\subset"{description}, draw=none, from=1-1, to=1-2]
    	\arrow["\supset"{description}, draw=none, from=1-3, to=1-2]
    	\arrow["\supset"{description}, draw=none, from=2-3, to=2-2]
    	\arrow["\subset"{description}, draw=none, from=2-1, to=2-2]
    	\arrow["\subset"{description}, draw=none, from=3-1, to=3-2]
    	\arrow["\supset"{description}, draw=none, from=3-3, to=3-2]
    \end{tikzcd}
    \caption{The two branched coverings of interest and their associ\-ated branch loci. The arrows marked $\longrightarrow\!\!\!\!\!\!\!\to$ denote an unbranched 2-fold cover from a 2-sphere to a real projective plane.}\label{fig:diagram}
\end{figure}

\section{The Arnold surface of an odd degree flexible curve}\label{sec:arnold}
    For a flexible curve $F\subset\CP^2$, let $A^+(F)=F/\conj=p(F)$. It is an embedded surface in $\bS^4$ with boundary $\partial A^+(F)\subset\cR$ identified with $\RR F$, and it is orientable\linebreak
if and only if $F$ is type I.

If the curve has \textit{even} degree, then $\RR F$ is null-homologous, and thus exactly one component of $\RP^2\smallsetminus\RR F$ is non-orientable (it is a punctured Möbius band). Let $\RP^2_\pm$ be the closure of the two possible subsets of $\RP^2\smallsetminus\RR F$ that have $\partial\RP^2_\pm=\partial F$. We choose $\RP^2_-$ to be the one containing the punctured Möbius band (i.e. $\RP^2_+$ is orientable, and $\RP^2_-$ has exactly one non-orientable component). In the case where $F$ is an \textit{algebraic} curve of even degree, the polynomial $P$ defining it can be chosen in such a way that $$\RP^2_\pm=\{[x_0:x_1:x_2]\in\RP^2\mid \pm P(x_0,x_1,x_2)\geqslant0\}.$$
In \cref[Figure]{fig:RP2plus}, we depict such an example for an algebraic curve.

\begin{figure}[ht]
    \centering
    \includegraphics[width=.3\textwidth]{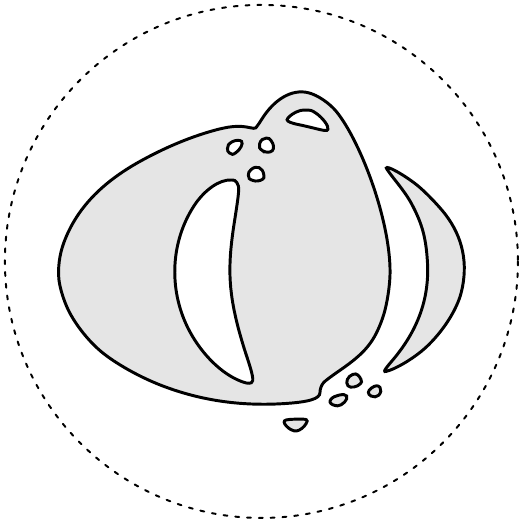}
    \caption{The set $\RP^2_+$, shaded, for Gudkov's $M$-sextic.}
    \label{fig:RP2plus}
\end{figure}

\begin{definition}
    Given a flexible curve $F$ of even degree, we let $$\cA(F)=A^+(F)\cup p(\RP^2_+)\subset\bS^4,$$
    and we call it the \textit{Arnold surface} of $F$.
\end{definition}

In odd degrees, it is not possible to define a surface in this way; we need to go to $\CPb^2$ first. Take $F$ to be a flexible curve of odd degree. We denote as $J\subset\RP^2$ the non-contractible component of $\RR F$ and as $o$ an oval of $\RR F$. Let ${J^+=p(J)}$,\linebreak
$o^+=p(o)$, and set $\overline{J}=\tilde{p}^{-1}(J^+)\subset\Qb$ and $\overline{o}=\tilde{p}^{-1}(o^+)\subset\Qb$. Observe that $\tilde{p}$\linebreak
restricts to an unbranched 2-fold covering $\tilde{p}:\overline{J}\to J^+$ and $\tilde{p}:\overline{o}\to o^+$.

\begin{proposition}
    We have $\overline{J}\cong\bS^1$ and $\overline{o}\cong\bS^1\sqcup\bS^1$, and the unbranched coverings $\tilde{p}:\overline{J}\to J^+$ and $\tilde{p}:\overline{o}\to o^+$ are respectively the non-trivial and the trivial 2-fold coverings of the circle.
\end{proposition}

\begin{proof}
    Isotope $J$ in $\RP^2$ to be $J=\RR X$ with $X\in\RR[x_0:x_1:x_2]$ a degree-one\linebreak
    nonsingular homogeneous polynomial. Note that we do not need to look at what happens \textit{outside} of $\RP^2$ for the claim. In particular, $\CC X\tcap Q$ is two points. Letting $G^+=p(\CC X)$ and $\overline{G}=\tilde{p}^{-1}(G^+)$, we obtain $$J^+=\partial G^+\quad\text{and}\quad\overline{J}=\partial\overline{G}.$$
    Moreover, $G^+\tcap\cQ$ is one point, for the two points in $\CC X\tcap Q$ are swapped pairwise by conjugation. In particular, the covering $\tilde{p}:\overline{G}\to G^+$ is a 2-fold branched cover of the disc $G^+$ (for in degree one, $\CC X$ is a sphere and $\RR X$ is type I), with one\linebreak
    branch point in its interior. This is unique, and it is known to induce the non-trivial cover on the boundary, so the first claim follows (see \cref[Figure]{fig:branching-ovals-and-J}).

    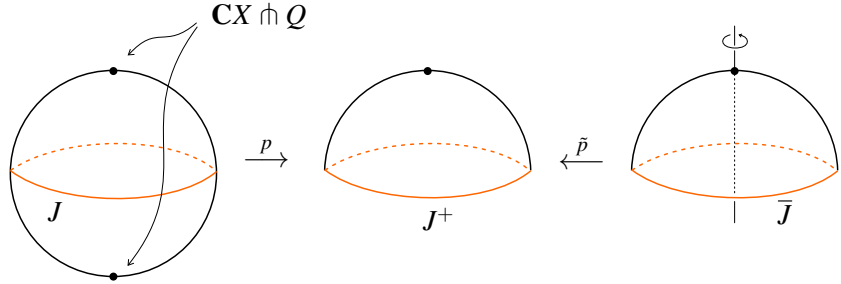
\begin{figure}[t]
        \centering
        \begin{tikzpicture}
            \node[inner sep=0pt] at (0,0) {\includegraphics{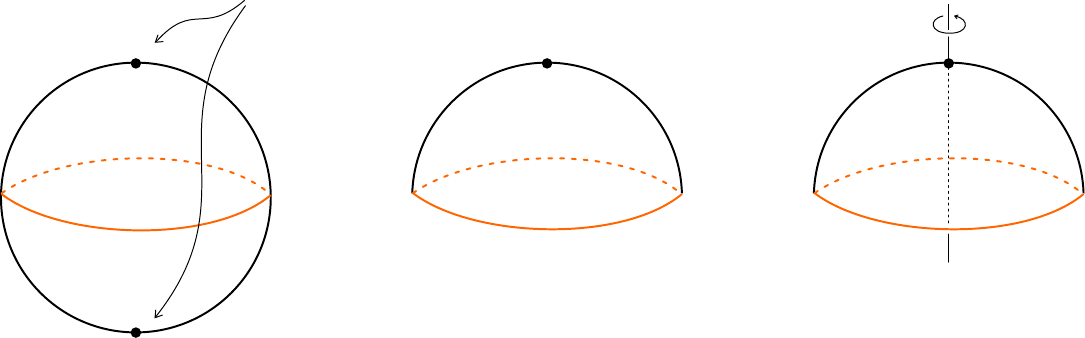}};
            \node at (-2.1,0) {$\overset{p}{\longrightarrow}$};
            \node at (2.1,0) {$\overset{\tilde{p}}{\longleftarrow}$};
            \node at (-4.9,-.8) {$J$};
            \node at (-2.2,1.8) {$\CC X\tcap Q$};
            \node at (.2,-.9) {$J^+$};
            \node at (4.8,-.8) {$\overline{J}$};
        \end{tikzpicture}
        \caption{The restrictions $p:\CC X\to G^+$ and $\tilde{p}:\overline{G}\to G^+$, the\linebreak
        second one being a branched covering. On the right, the rotation by 180$^\circ$ generates the group of deck transformations.}
        \label{fig:branching-ovals-and-J}
    \end{figure}
    
    For the other claim, an oval $o$ bounds a disc $D$ embedded in $\RP^2$, and is thus disjoint from $Q$. Therefore, the disc $D/\conj\subset\cR$ bounded by $o^+$ lifts in $\CPb^2$ to\linebreak
    two disjoint discs in $\Qb$. This means that $p:\bar{o}\to o^+$ is the trivial covering, and $\bar{o}$ is two circles.
\end{proof}

We let $\RRb F$ be the set $$\overline{J}\cup\bigcup_{o\text{ oval}}\overline{o}\subset\Qb.\vspace{10pt}$$
The previous statement implies that every oval of $\RR F$ gets doubled in $\RRb F$, whereas the non-contractible component $J$ does not.

\begin{proposition}\label{prop:position-ovals}
    Let $o_1$ and $o_2$ be ovals of $\RR F$.
    \begin{itemize}[leftmargin=21pt,itemsep=3pt]
        \item[\textup{(i)}] The set $\Qb\smallsetminus\overline{J}$ is two open discs, each containing one of the two components\linebreak
        of $\overline{o}_1$.
        \item[\textup{(ii)}] If $o_1\subset o_2$ (where inclusion means that $o_1$ is contained in the orientable\linebreak
        component of $\RP^2\smallsetminus o_2$), then $\overline{o_1}\subset\overline{o_2}$, in the following sense: $\Qb\smallsetminus\overline{o_2}$ has three components, one being a cylinder containing $\overline{J}$, and the other two being discs each containing a component of $\overline{o_1}$.
    \end{itemize}
\end{proposition}

\begin{proof}
    This comes from the observation that the covering $\tilde{p}:\Qb\to\cQ$ is the quotient of the $2$-sphere $\Qb$ by a fixed-point free involution (that is, the antipodal map), as well as the fact that $p:(\RP^2,\RR F)\to(\cR,p(\RR F))$ is a diffeomorphism of the pair.
\end{proof}

\begin{figure}[b]
    \centering
    \begin{tikzpicture}
        \node[inner sep=0pt] at (0,0) {\includegraphics[width=.55\textwidth]{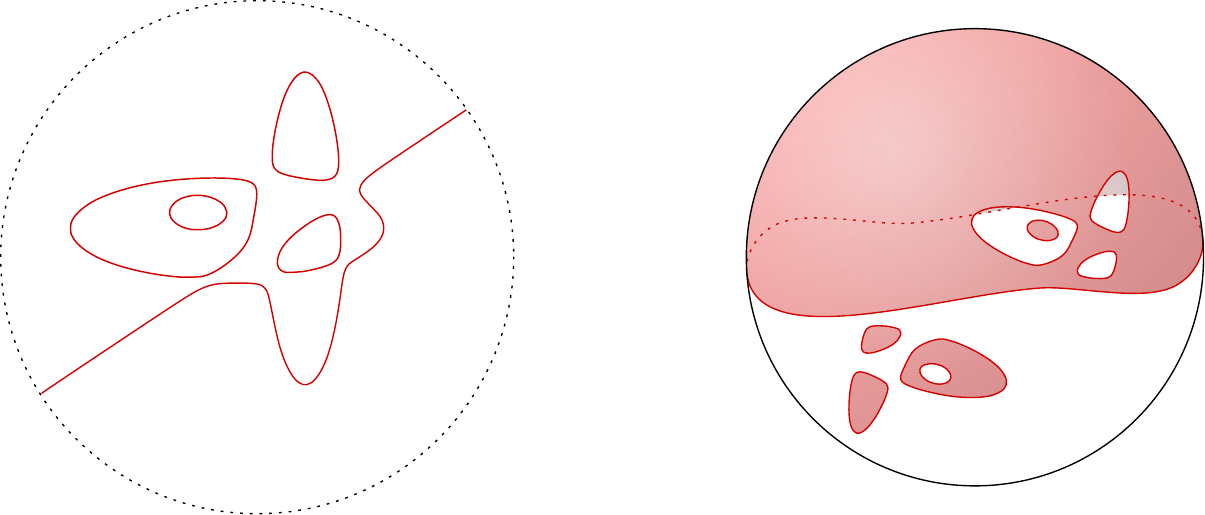}};
        \node at (.2,0) {$\overset{\tilde{p}}{\longleftarrow}$};
        \node at (-1.95,-1.85) {$p(\RR F)\subset\cR$};
        \node at (2.25,-1.85) {$\RRb F\subset\Qb$};
    \end{tikzpicture}
    \caption{The set $\Qb_+$, shaded, for an algebraic curve of degree 7 with real scheme $\langle J\sqcup 2\sqcup1\langle1\rangle\rangle$ (in Viro notation), obtained as a perturbation of three ellipses and a line.}
    \label{fig:arnold-surface}
\end{figure}
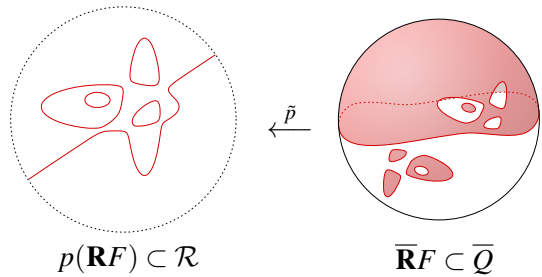

This means that the real scheme $\RR F$ can be seen \textit{doubled} in $\RRb F$, as \cref[Figure]{fig:arnold-surface} depicts. Now, define $\overline{A}^+(F)=\tilde{p}^{-1}(A^+(F))=\tilde{p}^{-1}(F/\conj)$. For the analogue\linebreak
of $\RP^2_+$, there are two subsets $\Qb_\pm$ of $\Qb\smallsetminus\RRb F$ that have $\partial\Qb_\pm=\RRb F$, and those are diffeomorphic, exchanged by "symmetry" of $\Qb$ along $\overline{J}$. To be more precise, we denote as $\Qb_\pm$ the closure of these two sets, with a choice involved in labeling one $\Qb_+$ and the other $\Qb_-$.

\begin{definition}
    The \textit{Arnold surface} of a flexible curve $F$ of odd degree is the surface $\cA(F)=\overline{A}^+(F)\cup\Qb_+\subset\CPb^2$.
\end{definition}

\section{Proving the inequality}\label{sec:proof}
    \noindent
    The idea is that we would like to take the 2-fold branched cover of $\CPb^2$ ramified along the Arnold surface. This is not yet possible in this odd degree setting, for the surface $\cA(F)$ is not null-homologous in $H_2(\CPb^2;\ZZ/2)$ (we will see that it has an odd self-intersection number). In fact, this limitation is what led Viro and Zvonilov to consider $h(m)$-sheeted branched covers, where $h(m)$ denotes the highest prime power that divides $m$. However, in our favorable setting, we can perturb the Arnold surface, with the important feature that it preserves the structure of the curve $\RRb F$ inside $\Qb$. One last remark is that we could \textit{not} apply the same construction to~a\linebreak
    $Q$-flexible curve $F\subset\CP^2$ directly, because the conic $Q$ has an \textit{even} homology class.
    \subsection*{Branching over the Arnold surface}

We are first interested in computing the\linebreak
normal Euler number of $\cA(F)\subset\CPb^2$. Recall that if $F\subset X$ is a closed surface in~a closed oriented $4$-manifold, then the Euler class $e(\nu F)\in H^2(F;\ZZ_w)$ corresponds to the self-intersection of $F$ (here, $\ZZ_w$ means coefficients twisted by $w_1(\nu F)$, and $w_1(X)=0$ implies $w_1(\nu F)=w_1(F)$, from which twisted Poincaré duality readily gives $H^2(F;\ZZ_w)\cong\ZZ$).

In the case where $\partial F\neq\varnothing$ however, one needs to choose a fixed nonvanishing section $\theta$ of $\nu F|_{\partial F}$, and consider a relative Euler class (see \cite{Sha73}): $$e_\theta(X,F)=e_\theta(\nu F)\in H^2(F,\partial F;\ZZ_w)\cong\ZZ.$$
This Euler class corresponds to the integer obstruction to extend this section $\theta$ to the whole $\nu F$. However, if one needs to glue two surfaces $F_1$ and $F_2$ along their common boundary $\partial F_1=\partial F_2$ and compute the Euler number of $F_1\cup_\partial F_2$ in terms of relative Euler numbers of $F_1$ and $F_2$, there are two things to be careful about:
\begin{enumerate}[leftmargin=21pt,itemsep=3pt]
    \item The bundle $\Lambda=(\nu F_1\cap\nu F_2)|_{\partial F_i}$ over $\partial F_i$ needs to be rank one.
    \item This bundle $\Lambda$ needs to have a non-vanishing section $\theta$.
\end{enumerate}
If both conditions are satisfied, the section $\theta$ gives rise to the \textit{same} section of\linebreak
$\nu F_1|_{\partial F_1}$ and $\nu F_2|_{\partial F_2}$. This can be used to define relative Euler numbers $e_\theta(X,F_i)$. Since, in the closed case, the number $e(X,F)$ does not depend on the choice of the (possibly vanishing) global section of $\nu F$, we obtain the relation $$e(X,F_1\cup_\partial F_2)=e_\theta(X,F_1)+e_\theta(X,F_2).$$
For instance, if $F\subset\CP^2$ is a flexible curve of even degree $m=2k$ with nonempty real part $\RR F$, one sees that $\Lambda=(\nu\RP^2_+\cap\nu F)|_{\RR F}$ is the trivial line bundle over $\RR F$. If $\theta$ denotes a section of the normal bundle $\RR F$ in $\RP^2$, then ${\bf i}\theta$ is a section of $\Lambda$,\linebreak
and letting $\cR_+=\RP^2_+/\conj$ and $A^+(F)=F/\conj$, it also induces a section $\hat{\theta}$~of $(\nu\cR_+\cap\nu A^+(F))|_{\partial A^+(F)}$. By a careful examination, one can use \cref[Lemma]{lem:branched-cover-euler-number} to compute $$e_{\hat{\theta}}(\bS^4,A^+(F))=\frac{1}{2}e_{{\bf i}\theta}(\CP^2,F)=\frac{1}{2}F\cdot F=2k^2,$$
because $F$ is closed,
and $$e_{\hat{\theta}}(\bS^4,\cR_+)=2e_{{\bf i}\theta}(\CP^2,\RP^2_+)=-2\chi(\RP^2_+),$$
because $\RP^2$ is Lagrangian. This means that the Arnold surface $\cA(F)\subset\bS^4$ has normal Euler number $$e(\bS^4,\cA(F))=2k^2-2\chi(\RP^2_+).$$

If $F\subset\CP^2$ is now a flexible curve of \textit{odd} degree, the normal bundle of $\RR F$ in $\RP^2$ is a non-trivial line bundle over $\RR F$ (to be more precise, exactly one connected\linebreak
component of this bundle is the non-orientable line bundle over the circle: the\linebreak
component associated to the pseudo-line $J\subset\RR F$).\hspace{-.05em} As such,\hspace{-.05em} there is no nonvanishing\linebreak
section $\theta$ of $\Lambda$, and it does not give rise to a section ${\bf i}\theta$ of $\nu F|_{\RR F}$. However, the\linebreak
subbundle ${\bf i}\Lambda\subset\nu F|_{\RR F}$ can be seen as a field of lines of $\nu F|_{\RR F}$ (instead of a section being a vector field).

In general, let $\Lambda\subset\nu F|_{\partial F}$ be a line sub-bundle. As done in \cite[\S3]{GM80}, one can still consider the integer obstruction $$\tilde{e}_{\Lambda}(X,F)\in H^2(F,\partial F;\ZZ_w)$$
to extend this field of line to the whole $\nu F$. In the case where $\Lambda$ \textit{does} have a\linebreak
section $\theta$, we have $\tilde{e}_{\Lambda}(X,F)=2e_\theta(X,F)$.

Back to where $F$ is a flexible curve of odd degree, and letting $A^+(F)=F/\conj$, we see that ${\bf i}\Lambda$ induces a line sub-bundle $\hat{\Lambda}$ of $\nu A^+(F)|_{\partial A^+(F)}$. From an application of \cref[Lemma]{lem:branched-cover-euler-number}, $$\tilde{e}_{\hat{\Lambda}}(\bS^4,A^+(F))=\frac{1}{2}\tilde{e}_{{\bf i}\Lambda}(\CP^2,F)=\frac{1}{2}\cdot 2e(\CP^2,F)=m^2,$$
because $F$ is closed. This means that, in the above sense, we have $e(\bS^4,A^+(F))=m^2/2$, although this is a non-integer value.

To ease out the exposition, we will allow ourselves to write half-integer Euler numbers and to use \cref[Lemma]{lem:branched-cover-euler-number} with half-integers. It will be understood that we use the obstruction $\tilde{e}$ when needed. We will also omit the choice of the field of lines in the subscript, as all surfaces will ultimately become closed at the end of computations.

\begin{proposition}
    We have $e(\CPb^2,\cA(F))=m^2-2$.\vspace{-.3em}
\end{proposition}

\begin{proof}
    Recall that we defined $\overline{A}^+(F)=\tilde{p}^{-1}(A^+(F))$, and $\cA(F)=\overline{A}^+(F)\cup\Qb_+$. Using \cref[Lemma]{lem:branched-cover-euler-number} twice, we compute that $e(\CPb^2,\overline{A}^+(F))=m^2$. Now, we simply make use of the fact that $e(\CPb^2,\Qb_+)=-2$. Indeed, $\Qb=\Qb_+\cup\Qb_-$, thus, $$-4=e(\CPb^2,\Qb)=e(\CPb^2,\Qb_+)+e(\CPb^2,\Qb_-),$$
    and because $\Qb_+$ and $\Qb_-$ are swapped by the (orientation-preserving) involution of $\CPb^2$ spanning $\Aut(\tilde{p})$, we obtain $$e(\CPb^2,\Qb_+)=e(\CPb^2,\Qb_-),$$
    from which we derive $e(\CPb^2,\Qb_\pm)=-2$. Alternatively, this can be obtained from the following lemma.\vspace{-.3em}
\end{proof}

\vspace{-.3em}
\begin{lemma}\label{lem:conic-lagrangian}
    Let $X$ be a submanifold of $\Qb$. Then $e(\CPb^2,X)=-2\chi(X)$.\vspace{-.3em}
\end{lemma}

\begin{proof}
    The submanifold $\RP^2\subset\CP^2$ being Lagrangian, and the covering $p:\CP^2\to\bS^4$ being branched exactly on $p(\RP^2)$, we observe that $\nu\cR\cong-T\cR$ in $\bS^4$. However, the covering $\tilde{p}:\CPb^2\to\bS^4$ is \textit{unbranched} in a regular neighborhood of $\cR$, whence $\nu\Qb\cong-2T\Qb$. In particular, for the Euler classes, we have $e(\CPb^2,X)=e(\nu X)=-2e(TX)=-2\chi(X)$.\vspace{-.2em}
\end{proof}

Because $\cA(F)$ has an odd self-intersection, we see that it cannot be null-\linebreak
homologous in $H_2(\CPb^2;\ZZ/2)$. In fact, because this group has rank one, being $\ZZ/2$-null-homologous is equivalent to having an even self-intersection. There is\linebreak
another surface, however, which is not null-homologous and transverse to $\cA(F)$:~the\linebreak
surface $\RPb^2$. If $F$ is a $Q$-flexible curve of odd degree $m$, the transverse intersection\linebreak
$F\tcap Q$ is $2m$ points. This implies that $\cA(F)$ intersects $\RPb^2$ transversely in $m$ points. The surface $\cA(F)\cup\RPb^2$ is therefore immersed with $m$ transverse crossings only. We will describe how to resolve those double points to obtain an embedded surface.

Firstly, in a closed oriented $4$-manifold $X$, let $\Sigma\subset X$ be the image of a closed\linebreak
surface through an immersion, with only one transverse self-intersection point\linebreak
$x\in\Sigma$. Take $B\subset X$ to be a small $4$-ball around $x$, which meets $\Sigma$ in two disks intersecting transversely at their common center $x$. The boundary of those discs is a Hopf link $\partial B\cap\Sigma\subset\partial B\cong\bS^3$, which bounds a Hopf band $H\subset B$. We call the\linebreak
surface $\Sigma'$ defined by a choice of a gluing of a Hopf band $H$ to $\Sigma\smallsetminus B$ a \textit{smoothing of the singularity} of the immersed surface $\Sigma\subset X$.

\vspace{-.3em}
\begin{lemma}\label{lem:res-sing}
    The resulting surface $\Sigma'$ is an embedded surface in $X$ with $\chi(\Sigma')=\chi(\Sigma)-1$, $e(X,\Sigma')=e(X,\Sigma)\pm2$, and we have freedom in the choice.\vspace{-.3em}
\end{lemma}

\begin{proof}
    Regarding the claim about the normal Euler numbers, we use similar\linebreak
    arguments as in \cite[\S5]{Yam95}. Note that if $B$ is a small $4$-ball around the double point $x\in\Sigma$, then the Hopf link $\partial B\cap\Sigma$ comes with two possible choices of orientations. Each determines a unique (up to isotopy fixing the boundary) oriented Hopf band $H$ inducing that orientation. A transverse push-off $s(\Sigma)$ of $\Sigma$ can be\linebreak
    assumed to be parallel to $\Sigma$ near $x$, and the intersection $s(\Sigma)\cap\Sigma\cap H$ is two points with the same sign. Finally, we see that those signs are opposite to one another in both choices of orientations of $\partial B\cap\Sigma$ (see \cref[Figure]{fig:sing-res}).
    
    The fact that $\chi(\Sigma')=\chi(\Sigma)-1$ is simply a matter of using the formula $\chi(A\cup B)=\chi(A)+\chi(B)-\chi(A\cap B)$ twice (here, all the sets involved are cellular subspaces). Indeed, if $H$ denotes the Hopf band that is glued to $\Sigma\smallsetminus B$, then $$\chi(\Sigma')=\chi(\Sigma\smallsetminus B\cup_\partial H)=\chi(\Sigma\smallsetminus B)+\chi(H)-\chi(\bS^1\sqcup\bS^1)=\chi(\Sigma\smallsetminus B),$$
    and $$\chi(\Sigma)=\chi(\Sigma\smallsetminus B\cup_\partial B\cap\Sigma)=\chi(\Sigma\smallsetminus B)+\chi(B\cap\Sigma)-\chi(\bS^1\sqcup\bS^1)=\chi(\Sigma\smallsetminus B)+1,$$
    by noting that $B\cap\Sigma$ is topologically a wedge of two discs.
\end{proof}

\begin{figure}[t]
    \centering
    \begin{tikzpicture}
        \node[inner sep=0pt] at (0,0) {\includegraphics{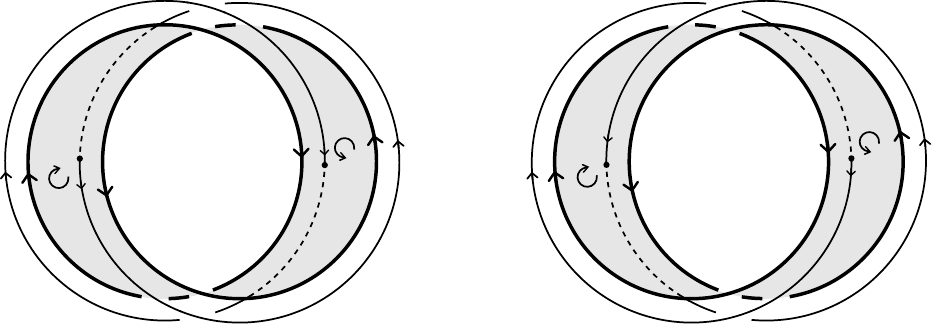}};
        \node at (-2.7,-2.2) {$e(\Sigma')=e(\Sigma)\pm2$};
        \node at (2.7,-2.2) {$e(\Sigma')=e(\Sigma)\mp2$};
    \end{tikzpicture}
    \caption{The two possible smoothings of a singularity of an\linebreak
    immersion, given by both choices of orientation of the Hopf link.}
    \label{fig:sing-res}
\end{figure}

Consider $F\subset\CP^2$ a $Q$-flexible curve of odd degree $m$. The Arnold surface $\cA(F)$ needs not be orientable, and as said before, there is no 2-fold branched cover of $(\CPb^2,\cA(F))$. Recall that $\cA(F)\tcap\RPb^2$ is $m$ points, and as such, $\cA(F)\cup\RPb^2$ is an immersed surface with $m$ double points. Applying the previous smoothing of\linebreak
the singularities at each of those $m$ points, this yields a surface $\cX(F)\subset\CPb^2$, with $$\chi(\cX(F))=\chi(\cA(F)\cup\RPb^2)-m;\qquad\qquad\qquad\qquad$$
$$e(\CPb^2,\cX(F))=e(\CPb^2,\cA(F)\cup\RPb^2)+2r,\quad r\in\{-m,\dots,m\}.$$
Here, $r$ is not free to take \textit{all} the possible values in $\{-m,\dots,m\}$. However, the extremal values $\pm m$ are always realizable. Define $\cX(F)$ to be the one where we pick up a $+2$ every time (that is, $r=+m$). Two applications of the topological Riemann--Hurwitz formula give $\chi(\cA(F))=\chi(F)-m+1$. Therefore, we have $$\chi(\cX(F))=-m^2+2.$$
Next, we compute $$e(\CPb^2,\cX(F))=e(\CPb^2,\cA(F)\cup\RPb^2)+2m=m^2+2m-1.$$

Take $Y^4$ to be the 2-fold cover of $\CPb^2$ branched over $\cX(F)$. This has been made possible because the surface $\cX(F)$ has zero homology mod 2: $[\cX(F)]=\linebreak0\in H_2(\CPb^2;\ZZ/2)$ (see \cite[\S6.3]{GS99} or \cite[Corollary 2.10]{Nag00}). Indeed, $$H_2(\CPb^2;\ZZ/2)=\{0,[\RPb^2]\},$$
and $\cA(F)$ intersects $\RPb^2$ in an odd number $m$ of points. Therefore, we deduce\linebreak
that $[\cA(F)]=[\RPb^2]$ in $H_2(\CPb^2;\ZZ/2)$. Now, adding $\RPb^2$ and smoothing the\linebreak
singularities means that $\cX(F)\cap\RPb^2=\varnothing$, and as such $[\cX(F)]=0$ in $H_2(\CPb^2;\ZZ/2)$. We denote as $\Theta:Y^4\to(\CPb^2,\cX(F))$ the 2-fold branched cover. The previous\linebreak
computations of $\chi(\cX(F))$ and $e(\CPb^2,\cX(F))$ will allow us to obtain homological information about the 4-manifold $Y$.

\begin{proposition}\label{prop:torsionHomology}
    The homology groups $H_1(Y;\ZZ)$ and $H_3(Y;\ZZ)$ are torsion. In particular, for Betti numbers, we have $b_1(Y)=b_3(Y)=0$.
\end{proposition}

\begin{proof}
    In order to show that $H_1(Y;\ZZ)$ is torsion, it is sufficient to know that\linebreak
    $H_1(Y;\ZZ/2)=0$, for any free part $\ZZ^p<H_1(Y;\ZZ)$ would give $p$ copies of $\ZZ/2$ in $H_1(Y;\ZZ/2)$. We use a generalisation of the Gysin sequence, as stated in \cite[Theorem 1]{LW95}: $$H_1(\CPb^2,\cX(F);\ZZ/2)\to H_1(Y,\ast;\ZZ/2)\to H_1(\CPb^2,\cX(F);\ZZ/2).$$
    Here, $H_1(Y,\ast;\ZZ/2)\cong\tilde{H}_1(Y;\ZZ/2)$ the reduced homology group, and we have\linebreak
    $H_1(\CPb^2,\cX(F);\ZZ/2)=0$, by looking at the homology long exact sequence of the pair $(\CPb^2,\cX(F))$. This provides $H_1(Y;\ZZ/2)=0$, as claimed. For $b_3(Y)=0$, this is a consequence of $b_1(Y)=0$ and Poincaré duality.
\end{proof}

An educated guess is that $Y$ \textit{may} be simply-connected, just like the usual branched cover of $\CP^2$ branched over an \textit{algebraic} curve $\{P(x_0,x_1,x_2)=0\}$, given as the\linebreak
algebraic surface $\{P(x_0,x_1,x_2)=w^2\}\subset\CP(1,1,1,m/2)$, is simply connected\linebreak
(see \cite{Wil78}). However, we have enough information to compute all the homological invariants of $Y$ that will be useful. We recall the Hirzebruch formula for the signature of 2-fold branched covers.

\begin{theorem}[{\cite[Section 3]{Hir69}}]\label{thm:H}
    Let $f:(Y,B)\to(X,A)$ be a cyclic 2-fold branched cover, with $X$ and $Y$ both closed oriented 4-manifolds, $A$ a closed surface and $f$ orientation-preserving. Then, we have $$\sigma(Y)=2\sigma(X)-\frac{1}{2}e(X,A).$$
\end{theorem}

\begin{proposition}
    We have
    $$\chi(Y)=m^2+4,\quad b_2(Y)=m^2+2,\quad \sigma(Y)=\frac{-m^2-2m-3}{2},$$
    \vspace{-.7em}
    $$b_2^+(Y)=\frac{(m-1)^2}{4}\quad\text{and}\quad b_2^-(Y)=\frac{3m^2+2m+7}{4},$$
    where $b_2^+(Y)$ and $b_2^-(Y)$ respectively denote the maximal ranks of the subspaces of $H_2(Y;\ZZ)$ on which the intersection form $Q_Y$ is positive and negative definite.
\end{proposition}

\begin{proof}
    First, the topological Riemann--Hurwitz formula again yields $$\chi(Y)=2\chi(\CPb^2)-\chi(\cX(F))=m^2+4.$$
    Next, we use the Theorem \ref{thm:H} with the branched cover $\Theta$ to obtain $$\sigma(Y)=2\sigma(\CPb^2)-\frac{1}{2}e(\CPb^2,\cX(F))=-2-\frac{m^2+2m-1}{2}=\frac{-m^2-2m-3}{2}.$$
    Now, because of \cref[Proposition]{prop:torsionHomology}, we see that $\chi(Y)=2+b_2(Y)$. This provides $$b_2^+(Y)+b_2^-(Y)=b_2(Y)\quad\text{and}\quad b_2^+(Y)-b_2^-(Y)=\sigma(Y),$$
    which we can easily solve for $b_2^\pm(Y)$.
\end{proof}
    \subsection*{Proving the inequality}

We will now mostly mimic the proof of Viro and Zvonilov \citeyear{VZ92}. Note that the construction of $\cX(F)$ from $\cA(F)\cup\RPb^2$ happens away from a neighborhood $\Qb$. In particular, we still see $\RRb F$ embedded inside $\cX(F)$. Given~an oval $o\subset\RR F$, recall that $\RP^2\smallsetminus o$ has two connected components, one of which is~a punctured disc (the other being a punctured Möbius band). Letting $C(o)\subset\cR$ be the image of that component under $p:\CP^2\to\bS^4$, we see that $\tilde{p}^{-1}(C(o))\subset\Qb$ is diffeomorphic to two disjoint copies of $C(o)$. We denote as $C_\pm(o)$ each of these copies, with the property that $C_\pm(o)\subset\Qb_\pm$ (see \cref[Figure]{fig:cPMoval}).

\begin{figure}[b]
    \centering
    \begin{tikzpicture}
        \node[inner sep=0pt] at (0,0) {\includegraphics{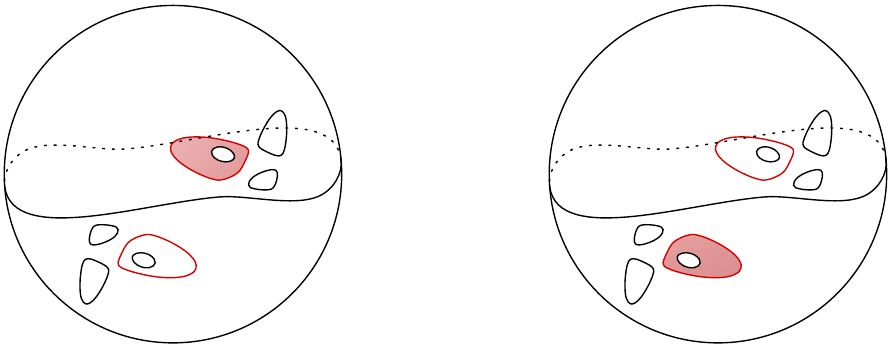}};
        \node at (-2.7,-2.3) {$C_-(o)\subset\Qb_-$};
        \node at (2.9,-2.3) {$C_+(o)\subset\Qb_+$};
    \end{tikzpicture}
    \caption{Using the same scheme $\langle J\sqcup 2\sqcup1\langle1\rangle\rangle$ as in the example of \cref[Figure]{fig:arnold-surface}, we take $o$ to be the only non-empty oval. In the shaded regions, we depict $C_\pm(o)$, where part of the boundary $\partial C_\pm(o)$ is $\overline{o}$.}
    \label{fig:cPMoval}
\end{figure}

We see that $C_+(o)$ is totally included in the ramification locus of the branched cover $\Theta:Y\to\CPb^2$, and that $C_-(o)$ intersects this ramification locus only at its boundary $\partial C_-(o)\supset\overline{o}$. We let $\tilde{C}(o)=\Theta^{-1}(C_-(o))$. The restriction $\Theta:\tilde{C}(o)\to C_-(o)$ is not a branched cover, but it is close enough: it maps the boundaries $\partial\tilde{C}(o)\to\partial C_-(o)$ diffeomorphically, and is two-to-one on the interior. Because $C_-(o)$ is planar (that is, a sphere with holes), we have that $\tilde{C}(o)$ is obtained as\linebreak
gluing two spheres with holes along their boundary components. Additionally, $\Aut(\Theta)$ is a $\ZZ/2$ spanned by $\tau:Y\to Y$ an orientation-preserving involution. This involution $\tau$ swaps those two planar surfaces in $Y$ that glue to $\tilde{C}(o)$ and fixes their common boundary. As such, we have shown the next result.

\begin{proposition}
    For any oval $o\subset\RR F$, $\tilde{C}(o)$ is an oriented surface in $Y$ of genus\linebreak
    $b$ the number of ovals directly contained in $o$. The restriction $\Theta:\tilde{C}(o)\to C_-(o)$, shown in \cref[Figure]{fig:pseudo-BC}, is the result of the quotient of the surface $\Sigma_b$ by reflection along a plane of symmetry that cuts it into two planar surfaces.\qed
\end{proposition}

\begin{figure}[t]
    \centering
    \begin{tikzpicture}
        \node[inner sep=0pt] at (0,0) {\includegraphics{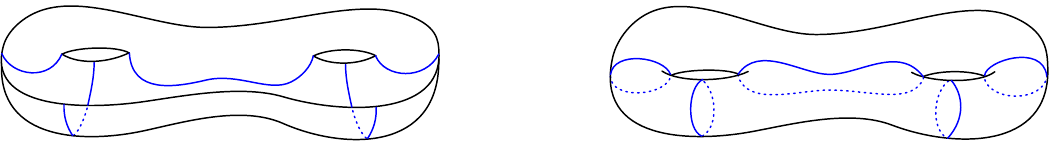}};
        \node at (0,0) {$\overset{\Theta}{\longleftarrow}$};
        \node at (-3,-1.1) {$\chi(o)=-1$};
        \node at (3,-1.1) {$e(Y,\tilde{C}(o))=+4$};
        \node at (-3,1.1) {$C_-(o)$};
        \node at (3,1.1) {$\tilde{C}(o)$};
    \end{tikzpicture}
    \caption{The ``pseudo'' branched cover $\tilde{C}(o)\to C_-(o)$.}\label{fig:pseudo-BC}
\end{figure}
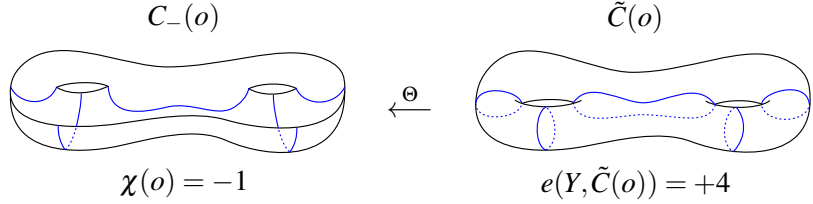

The same construction works for $J$: there are two path-connected subsets $D_\pm(J)$ of $\Qb_\pm$ that have $\overline{J}$ as a part of their boundary. Letting $\tilde{D}(J)=\Theta^{-1}(D_-(J))$, we have that $\tilde{D}(J)$ is a surface of genus $e$ the number of exterior ovals in $\RR F$ (those\linebreak
not included in any other), and the restriction $\Theta:\tilde{D}(J)\to D_-(J)$ is again the\linebreak
quotient of $\Sigma_e$ by reflection along a plane in the middle.

Given an oval $o\subset\RR F$, we denote as $\chi(o)=\chi(C_-(o))$ the Euler characteristic of the connected subset of $\RP^2\smallsetminus\RR F$ bounded by $o$ from outside. Similarly, we~let $\chi(J)=\chi(D_-(J))$. One remarks that $\chi(o)\leqslant1$, with equality if and only if $o$ is empty, and that $\chi(J)=1-e$ with $e$ the number of exterior ovals.

\begin{proposition}
    Let $o,o'\subset\RR F$ be ovals, and denote again by  $J$ the non-contractible component of $\RR F$.
    \begin{itemize}[leftmargin=21pt,itemsep=3pt]
        \item[\textup{(1)}] We have $Q_Y(\tilde{C}(o),\tilde{C}(o))=-4\chi(o)$ and $Q_Y(\tilde{D}(J),\tilde{D}(J))=-4\chi(J)=\linebreak4(e-1)$.
        \item[\textup{(2)}] We have $Q_Y(\tilde{C}(o),\tilde{D}(J))=0$. If $o\neq o'$, then $Q_Y(\tilde{C}(o),\tilde{C}(o'))=0$.
    \end{itemize}
\end{proposition}

\begin{proof}
    For the first claim, observe $e(\CPb^2,C_-(o))=-2\chi(o)$ and $e(\CPb^2,D_-(J))=-2\chi(J)$, by using \cref[Lemma]{lem:conic-lagrangian}. Next, from \cref[Lemma]{lem:branched-cover-euler-number}, we can see that $$e(Y,\tilde{C}(o))=2e(\CPb^2,C_-(o))\text{ and }e(Y,\tilde{D}(J))=2e(\CPb^2,D_-(J)).$$
    To derive $Q_Y(\tilde{C}(o),\tilde{C}(o))$ and $Q_Y(\tilde{D}(J),\tilde{D}(J))$, we remark that $\tilde{C}(o)$ and $\tilde{D}(J)$\linebreak
    are orientable surfaces, so the self-intersection and the evaluation of the intersection\linebreak
    form agree.
    
    For the second claim, distinct ovals $o$ and $o'$ cannot satisfy $C_-(o)\cap C_-(o')\neq\varnothing$, even if one is included inside the other (but, it is possible that $C_-(o)\cap C_+(o')\neq\varnothing$). The same goes for $C_-(o)\cap D_-(J)=\varnothing$. As such, the surfaces $\tilde{C}(o)$, $\tilde{C}(o')$ and\linebreak
    $\tilde{D}(J)$ are non-intersecting in $Y$.
\end{proof}

The homology classes of the surfaces $\tilde{C}(o_i)$, $i\in[\![1,\ell]\!]$, and $\tilde{D}(J)$ were respec\-tively denoted as $\beta_i$ and $\beta_0$ by Viro and Zvonilov (where $\ell$ denotes the number of ovals in $\RR F$). They showed the following result.

\begin{lemma}[{\cite[Lemma 1.3]{VZ92}}]\label{lem:VZ-Smith}
    Let $h=p^r$ be a prime power. Let $\nu:Y\to X$ be an $h$-sheeted cyclic covering between two $n$-manifolds, branched over a codimension-two subset $A\subset X$. Let $B\subset X$ be a membrane, let $b$ be the class in $H_k(X,A)$ determined by $B$, and let $\beta$ be the class in $H_k(Y)$ determined by $\nu^{-1}(B)$, oriented coherently with $B$. Let $\tau:Y\to Y$ be a generator of $\Aut(\nu)$, and let $\varrho=1-\tau\in(\ZZ/p)[\Aut(\nu)]$. Recall the Smith long exact sequence in homology (with coefficients in $\ZZ/p$):
    $$\cdots\to H_{k+1}^\varrho(Y)\overset{\partial}{\longrightarrow}H_k(X,A)\oplus H_k(A)\overset{\alpha_k}{\longrightarrow}H_k(Y)\overset{\varrho_\ast}{\longrightarrow}H_k^\varrho(Y)\to\cdots.$$
    Then, the restriction $\tilde{\alpha}_k:H_k(X,A)\to H_k(Y)$ maps $b$ to $\beta$, and
    \begin{itemize}[leftmargin=21pt,itemsep=3pt]
        \item[\textup{(1)}] $\alpha_{n-1}$ is monic if $H_n^\varrho(Y)=0$;
        \item[\textup{(2)}] $\tilde{\alpha}_{n-2}$ is monic if $X$ is connected and $H_{n-1}(Y)=0$;
        \item[\textup{(3)}] if $\lfloor(n+1)/2\rfloor\leqslant k<n-2$, then $\alpha_k$ is monic if $X$ and $A$ are connected and if $H_i(Y)=0$ for all $k+1\leqslant i\leqslant n-1$.
    \end{itemize}
\end{lemma}

We can now prove an analogue to their Corollary 1.5.C.

\begin{corollary}\label{cor:rank}
    The set $\{\tilde{C}(o_i)\mid1\leqslant i\leqslant\ell\}\cup\{\tilde{D}(J)\}$ has rank at least $\ell$ (where $\ell$\linebreak
    is the number of ovals $o_1,\dots,o_\ell$ of the curve). If the family has rank $\ell+1$, then\linebreak
    the curve is type I.
\end{corollary}

\begin{proof}
    We can apply \cref[Lemma]{lem:VZ-Smith} in our setting, where $\nu=\Theta:Y\to(\CPb^2,\cX(F))$ and $h=2$. We then see that $$\tilde{\alpha}_2:H_2(\CPb^2,\cX(F);\ZZ/2)\to H_2(Y;\ZZ/2)$$
    is injective, because $\CPb^2$ is connected and $H_3(Y;\ZZ/2)=0$ (\cref[Proposition]{prop:torsionHomology}). Noting that $\tilde{\alpha}_2(C_-(o))=\tilde{C}(o)$ and $\tilde{\alpha}_2(D_-(J))=\tilde{D}(J)$, the claim follows from the very same arguments as in \cite[\S2.4]{VZ92}.
\end{proof}

Recall that $\ell^\pm$ and $\ell^0$ denote the number of ovals of the curve that bound from~the outside a component of $\RP^2\hspace{-.05em}\smallsetminus\hspace{-.05em}\RR F$ with positive/negative or zero Euler \mbox{characteristic,}
respectively. The previous results finally wraps up to yield our main theorem.

\begin{theorem}\label{thm:main}
    Let $F$ be a $Q$-flexible curve of odd degree $m$. Then $$\ell^0+\ell^-\leqslant\frac{(m-1)^2}{4}.$$
    If equality holds, then the curve is type I.
\end{theorem}

\begin{proof}
    Take the maximal subset $\mathcal{P}$ of $\{\tilde{C}(o_i)\mid1\leqslant i\leqslant\ell\}\cup\{\tilde{D}(J)\}$ that spans a subspace of $H_2^+(Y)$, and let $r=\mathrm{rank}(\mathcal{P})$. Then, we obtain $r\leqslant b_2^+(Y)$. Moreover, because of $Q_Y(\tilde{C}(o),\tilde{C}(o))=-4\chi(o)$ and similarly for $\tilde{D}(J)$, observe that $\mathcal{P}$\linebreak
    has exactly $\ell^0+\ell^-+1$ elements (assuming that there is at least one oval to have $\tilde{D}(J)\in\mathcal{P}$; if there are none, the theorem is vacuous). Therefore, because of\linebreak
    \cref[Corollary]{cor:rank}, we deduce $r\geqslant\#\mathcal{P}-1=\ell^0+\ell^-$. This produces $$\ell^0+\ell^-\leqslant b_2^+(Y),$$
    which is the claimed inequality. The extremal case also follows from an almost word-for-word proof as in \cite{VZ92}.
\end{proof}

\section{Curves on a quadric}\label{sec:quadrics}\noindent
    We investigate our method for flexible curves in $\CP^1\times\CP^1$, with either of its\linebreak
    antiholomorphic involutions $c_1(x,y)=(\bar{x},\bar{y})$ or $c_2(x,y)=(\bar{y},\bar{x})$. This is motivated by recent work from Zvonilov \citeyear{Zvo22}, which generalizes \cite{VZ92} to flexible curves on almost-complex 4-manifolds. For a survey of results regarding curves in $\CP^1\times\CP^1$, we refer the reader to \cite{Mat91} or \cite{Gil91}.\linebreak
    We will also need the following result.

    \begin{theorem}{\cite[\S3]{Let84}}
        There are diffeomorphisms $\CP^1\times\CP^1/c_1\cong\bS^4$ and $\CP^1\times\CP^1/c_2\cong\CPb^2$.
    \end{theorem}

    More precisely, the differential structure on $\CP^1\times\CP^1\smallsetminus\mathrm{Fix}(c_i)/c_i$ extends to the standard one on $\bS^4$ or $\CPb^2$, respectively.

    Note that in the present work, we do not make any assumption regarding $\gcd(a,b)$ with $[F]=(a,b)\in H_2(\CP^1\times\CP^1)$, contrary to \cite{Zvo22} where there is\linebreak
    no result if $\gcd(a,b)=1$.
    
    \subsection*{Curves on a hyperboloid}

Consider the space $X=\CP^1\times\CP^1$ with its involution $c_1:([x_0:x_1],[y_0:y_1])\mapsto([\bar{x_0}:\bar{x_1}],[\bar{y_0}:\bar{y_1}])$. We call $(X,c_1)$ the \textit{hyperboloid}. Let $\fR=\mathrm{Fix}(c_1)=\RP^1\times\RP^1$. We consider $\fQ$ to be a generic real algebraic curve of bidegree $(2,2)$ and with empty real part $\RR\fQ=\varnothing\subset\fR$. We will prove the following result.

\begin{reptheorem}{thm:hyperboloid}
    Let $F$ be a $\fQ$-flexible curve in the hyperboloid with bidegree $(a,b)$ where both $a$ and $b$ are odd. Let $\ell^\pm$ and $\ell^0$ denote the number of ovals of the curve that bound from the outside a subset with positive, negative or zero Euler characteristic, respectively. Then $$\ell^-+\ell^0\leqslant\frac{ab+1}{2}.$$
\end{reptheorem}

Note that $H_2(X;\ZZ)$ is a $\ZZ\oplus\ZZ$ spanned by the homology classes of algebraic curves of bidegree $(1,0)$ and $(0,1)$. We have a notion of a ($\fQ$-)flexible curve in this setting too.

\begin{definition}
    Let $F\subset X$ be a closed, connected and oriented surface. We call $F$\linebreak
    a \textit{bidegree $(a,b)$ flexible curve} if the following conditions hold:\vspace{.2em}
    \begin{itemize}[leftmargin=21pt,itemsep=3pt]
        \item[\textup{(1)}] $\conj(F)=F$.
        \item[\textup{(2)}] $[F]=(a,b)$ in $H_2(X;\ZZ)=\ZZ\oplus\ZZ$.
        \item[\textup{(3)}] $\chi(F)=2-2(a-1)(b-1)$.
        \item[\textup{(4)}] If $\RR F=F\cap\fR$, then for all $x\in\RR F$, $T_xF=T_x\RR F\oplus{\bf i}\cdot T_x\RR F$.\vspace{.3em}
    \end{itemize}
    If, additionally, $F\tcap\fQ$ is $2(a+b)$ points, then $F$ is said to be \textit{$\fQ$-flexible}.
\end{definition}

Note that if both $a$ and $b$ are odd, then $\RR F$ is some number of ovals (null-homologous curves in $\fR$), and some non-zero number of parallel copies of a curve with homology class $(\alpha,\beta)$ in $H_1(\fR;\ZZ)\cong\ZZ\oplus\ZZ$, where $0\leqslant\alpha\leqslant a$ and $0\leqslant\beta\leqslant b$ are both odd and coprime, and $\pi_1(\fR)=H_1(\fR;\ZZ)\cong\ZZ\oplus\ZZ$ is spanned by the real parts of bidegree $(1,0)$ and $(0,1)$ algebraic curves. In the case of an oval $o$, the complement $\fR\smallsetminus o$ has two connected components, one of which is a disk and is called the \textit{interior} of that oval, and we say that $o$ bounds it \textit{from the outside}.

We observe that $\fR$ is a null-homologous torus, and $\fQ$ is a torus with homology class $(2,2)$, both in $H_2(X;\ZZ)$. Therefore $$e(X,\fR)=0\text{ and }e(X,\fQ)=8.$$
Denoting as $p:X\to X/c_1\cong\bS^4$ the 2-fold branched cover, we set $\cR=p(\fR)$\linebreak
and $\cQ=p(\fQ)$. Observe that $\cR$ is a torus and $\cQ$ is a Klein bottle. Finally, \mbox{letting}\linebreak
$\tilde{p}:\Xb\to\bS^4$ be the 2-fold branched cover of $\bS^4$ ramified along $\cQ$ (which exists\linebreak
because $[\cQ]=0\in H_2(\bS^4;\ZZ/2)\cong0$), we set $\fRb=\tilde{p}^{-1}(\cQ)$ and $\fQb=\tilde{p}^{-1}(\cR)$.\linebreak
Consecutive applications of \cref[Lemma]{lem:branched-cover-euler-number} yield $$e(\bS^4,\cR)=0,\quad e(\bS^4,\cQ)=4,\quad e(\Xb,\fRb)=2\quad\text{and}\quad e(\Xb,\fQb)=0.$$ The situation is depicted in \cref[Figure]{fig:hyperboloid}.

\begin{figure}[t]
    \centering
    \begin{tikzcd}
    	\fR & X & \fQ \\
    	\cR & {\bS^4} & \cQ \\
    	\fQb & \Xb & \fRb
    	\arrow["{\tilde{p}}"', from=3-2, to=2-2]
    	\arrow["p"', from=1-2, to=2-2]
    	\arrow["{||}"{description}, draw=none, from=1-1, to=2-1]
    	\arrow["{||}"{description}, draw=none, from=3-3, to=2-3]
    	\arrow[shorten <=1pt, shorten >=1pt, two heads, from=1-3, to=2-3]
    	\arrow[shorten <=2pt, shorten >=2pt, two heads, from=3-1, to=2-1]
    	\arrow["\subset"{description}, draw=none, from=1-1, to=1-2]
    	\arrow["\supset"{description}, draw=none, from=1-3, to=1-2]
    	\arrow["\supset"{description}, draw=none, from=2-3, to=2-2]
    	\arrow["\subset"{description}, draw=none, from=2-1, to=2-2]
    	\arrow["\subset"{description}, draw=none, from=3-1, to=3-2]
    	\arrow["\supset"{description}, draw=none, from=3-3, to=3-2]
    \end{tikzcd}
    \caption{The branched covers in the case of $\CP^1\times\CP^1$ with~its hyperboloid structure, with the same notation conventions as in \cref[Figure]{fig:diagram}.}\label{fig:hyperboloid}
\end{figure}

The topological Riemann--Hurwitz formula gives $\chi(\Xb)=4$, and \cref[Theorem]{thm:H} provides $\sigma(\Xb)=-2$. A similar reasoning as in \cref[Proposition]{prop:torsionHomology} ensures that\linebreak
$H_1(\Xb;\ZZ/2)=0$, and thus that $H_1(\Xb;\ZZ)$ is torsion. In particular, $$b_1(\Xb)=b_3(\Xb)=0\quad\text{and}\quad b_2(\Xb)=-\sigma(\Xb)=2.$$
This suggests that $\Xb$ may be diffeomorphic to $\CPb^2\#\CPb^2$, but this will not be needed.

Consider a $\fQ$-flexible curve $F\subset X$ of bidegree $(a,b)$, where $a$ and $b$ are both odd. In particular: $$\chi(F)=-2ab+2a+2b\quad\text{and}\quad e(X,F)=2ab.$$
Letting $A^+(F)=p(F)$ and $\overline{A}^+(F)=\tilde{p}^{-1}(A^+(F))$, one checks that $$\chi(\overline{A}^+(F))=-2ab+a+b\quad\text{and}\quad e(\overline{A}^+(F))=2ab.$$
In order to understand $\RRb F=\tilde{p}^{-1}(\partial A^+(F))\subset\fQb$, it is necessary to describe the unbranched 2-fold covering $\tilde{p}:\fQb\to\cR$, which is a non-trivial 2-fold cover of the torus (non-triviality can be deduced by the same argument as in the proof of the~next proposition). There are only three such coverings, each given by the subgroups $2\ZZ\oplus\ZZ$, $\ZZ\oplus2\ZZ$ and $G=\{(x,y)\in\ZZ^2\mid x+y\equiv0\mod{2}\}$.

\begin{proposition}
    The covering $\tilde{p}:\fQb\to\cR$ corresponds to the subgroup $G$.
\end{proposition}

\begin{proof}
    Assume it corresponds to the subgroup $2\ZZ\oplus\ZZ$ (the argument is the same with the other). Let $\gamma$ be a curve with homology class $(0,1)$ in $\cR$. Its pre-image is therefore two parallel copies of it. The situation is depicted in \cref[Figure]{fig:covering-torus}.
    
    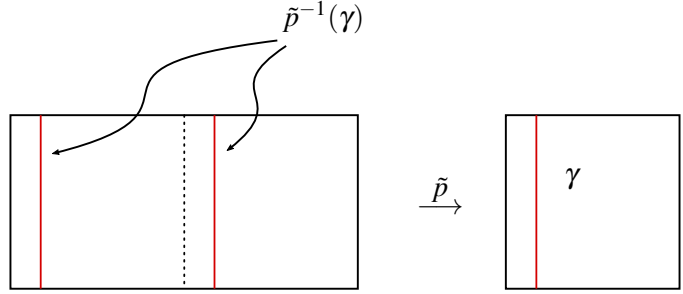
\begin{figure}[ht]
        \centering\vspace{-.5em}
        \begin{tikzpicture}
            \node[inner sep=0pt] at (0,0) {\includegraphics[width=.7\textwidth]{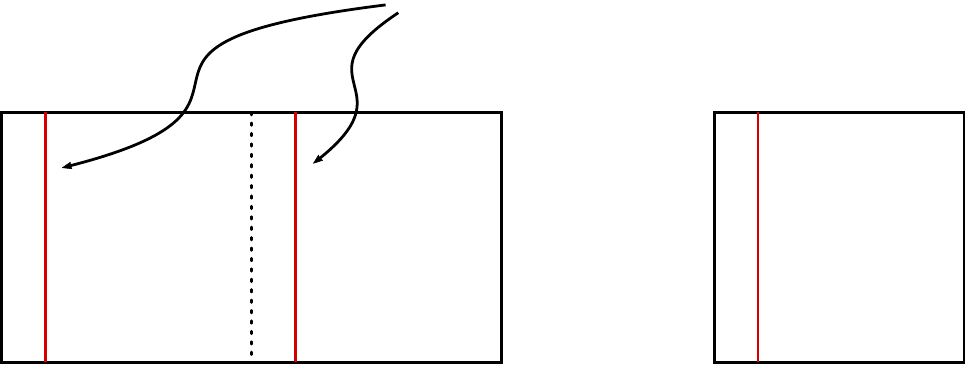}};
            \node at (-.3,1.95) {$\tilde{p}^{-1}(\gamma)$};
            \node at (1.25,-.6) {$\longrightarrow$};
            \node at (1.25,-.35) {$\tilde{p}$};
            \node at (3,-.2) {$\gamma$};
        \end{tikzpicture}
        \caption{The unbranched 2-fold covering of the torus corresponding to the subgroup $2\ZZ\oplus\ZZ$, and its effect on the curve with homology class $(0,1)$.\vspace{-1em}}\label{fig:covering-torus}
    \end{figure}

    Now, let $C$ be a generic bidegree $(0,1)$ algebraic curve, so that $\partial A^+(C)=\gamma$. Then $A^+(C)\tcap\cQ$ is one point, so that the map $\tilde{p}:\overline{A}^+(C)\to A^+(C)$ is a branched covering that restricts to an unbranched covering of the boundary. An application of the Riemann--Hurwitz formula gives $\chi(\overline{A}^+(C))=1$, and $\overline{A}^+(C)$ has at most two boundary components. Therefore, there is no other choice but the same situation as depicted in \cref[Figure]{fig:branching-ovals-and-J}. That is, $\overline{A}^+(C)$ is a disk, and $\tilde{p}^{-1}(\gamma)=\partial\overline{A}^+(C)$ is connected. This is excluded, by assumption.\hspace{-.05em} The same argument with a\hspace{-.05em} \mbox{bidegree $(1,0)$ algebraic} curve works for the subgroup $\ZZ\oplus2\ZZ$.\vspace{-.5em}
\end{proof}

The covering corresponding to the subgroup $G$ is depicted in \cref[Figure]{fig:true-covering-torus}. If a\linebreak
curve $\gamma\subset\cR$ has homology class $(\alpha,\beta)$ with both $\alpha$ and $\beta$ odd (and coprime), then its pre-image is two parallel copies $\tilde{p}^{-1}(\gamma)\subset\fQb$. 

\begin{figure}[b]
    \centering
    \begin{tikzpicture}
        \node[inner sep=0pt] at (0,0) {\includegraphics[width=.56\textwidth]{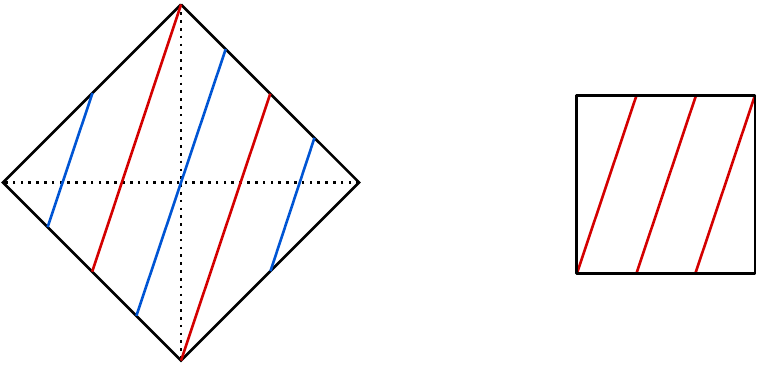}};
        \node at (.8,0) {$\longrightarrow$};
        \node at (.8,.25) {$\tilde{p}$};
    \end{tikzpicture}
    \caption{The unbranched covering $\tilde{p}:\fQb\to\cR$ corresponding to the subgroup $G=\{(x,y)\in\ZZ^2\mid x+y\equiv0\mod{2}\}\subset\ZZ^2$. On the left, the pre-image of the $(3,1)$ curve is two parallel copies of a $(1,-2)$ curve. It is understood that the two tori are represented by the two squares, whose opposite sides are identified.}\label{fig:true-covering-torus}
\end{figure}

Recalling that $\RR F$ is some ovals and some number of parallel copies of an $(\alpha,\beta)$ curve with $\alpha$ and $\beta$ coprime and odd, we have the following immediate facts:
\begin{enumerate}[leftmargin=21pt,itemsep=3pt]
    \item Each copy of the $(\alpha,\beta)$ curve is doubled (indeed, the homotopy class of that curve belongs to the subgroup $G$).
    \item Each oval is doubled.
    \item The preimage respects mutual position of ovals, as in \cref[Proposition]{prop:position-ovals} (that is, an oval inside another lifts to two copies inside the other two copies).
\end{enumerate}

Hence, we see that $\fQb\smallsetminus\RRb F$ has two diffeomorphic subsets $\fQb_\pm$ with the property $\partial\fQb_\pm=\RRb F$ (we provide an example in \cref[Figure]{fig:arnold-hyperboloid}). We therefore have $$\chi(\fQb_+)=\chi(\fQb_-)\quad\text{and}\quad\chi(\fQb)=\chi(\fQb_+)+\chi(\fQb_-),$$
so that $\chi(\fQb_\pm)=0$. The same argument gives $e(\Xb,\fQb_\pm)=0$.

\begin{figure}[t]
    \centering
    \begin{tikzpicture}
        \node[inner sep=0pt] at (0,0) {\includegraphics{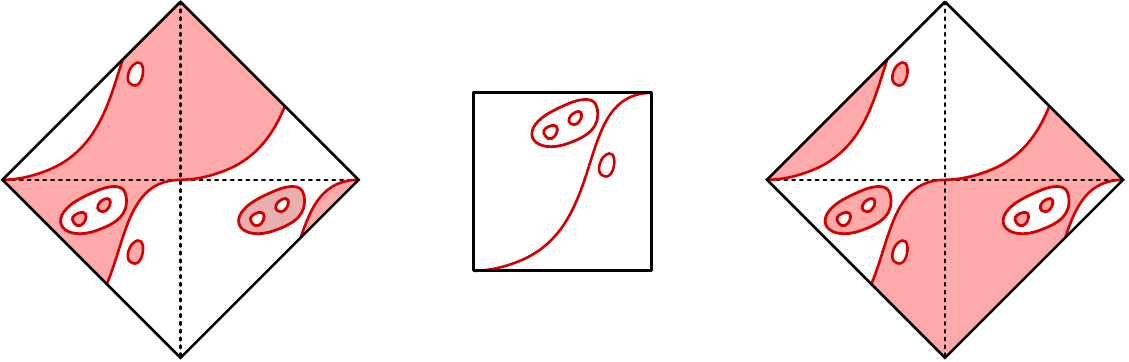}};
        \node at (1.4,0) {$\longleftarrow$};
        \node at (-1.5,0) {$\longrightarrow$};
    \end{tikzpicture}
    \caption{In the middle, a curve with real scheme $\langle(1,1),1\sqcup1\langle2\rangle\rangle$. On the left and on the right, the two possible choices $\fQb_\pm$.}\label{fig:arnold-hyperboloid}
\end{figure}

Therefore, we can define the \textit{Arnold surface} of the curve as $\cA(F)=\overline{A}^+(F)\cup\fQb_+$. Note that $\cA(F)\tcap\fRb$ is $a+b$ points (coming from the $2(a+b)$ points in $F\tcap\fQ$). We consider the immersed surface $\cA(F)\cup\fRb$, and we let $\cX(F)$ be the smoothing of its singularities, as provided by \cref[Lemma]{lem:res-sing}, where we choose the smoothing that satisfies $$e(\Xb,\cX(F))=e(\Xb,\cA(F)\cup\fRb)+2(a+b).$$

\begin{proposition}
    Let $F\subset X$ be a $\fQ$-flexible curve of bidegree $(a,b)$ with both $a$\linebreak
    and $b$ odd. The surface $\cX(F)$ has zero homology in $H_2(\Xb;\ZZ/2)$ and satisfies $$\chi(\cX(F))=-2ab-a-b\quad\text{and}\quad e(\Xb,\cX(F))=2ab+2a+2b+2.$$
\end{proposition}

\begin{proof}
    Computing $\chi(\cX(F))$ and $e(\Xb,\cX(F))$ is straight-forward. To prove that $[\cX(F)]=0\in H_2(\Xb;\ZZ/2)$, it suffices to show that $\cA(F)$ and $\fRb$ are homologous mod 2. Note that by the previous computations, $b_2(\Xb)=-\sigma(\Xb)=2$, so that $\Xb$ is a negative definite smooth 4-manifold. By virtue of Donaldson's theorem, this means that the intersection form of $\Xb$ is, up to a change of basis, that of $\CPb^2\#\CPb^2$. We consider a basis of $H_2(\Xb,\ZZ/2)\cong\ZZ/2\oplus\ZZ/2$ that diagonalizes this intersection form, and we will show that $\cA(F)$ and $\fRb$ both realize the homology class $(1,1)$ in $H_2(\Xb;\ZZ/2)$.
    
    Because $e(\Xb,\cA(F))$ and $e(\Xb,\fRb)$ are both even, this rules out the two classes $(1,0)$ and $(0,1)$. As such, it suffices to show that $\cA(F)$ and $\fRb$ are both not null-homologous in $H_2(\Xb;\ZZ/2)$ (if $\cA(F)$ was null-homologous, we could directly take the 2-fold covering of $\Xb$ ramified along $\cA(F)$, without adding $\fRb$).

    By \cref[Theorem]{thm:YamadaWhitney}, we have the congruences
    $$e(\Xb,\cA(F))+2\chi(\cA(F))\equiv q([\cA(F)])\mod{4},$$
    $$\hspace{1.5em}e(\Xb,\fRb)+2\chi(\fRb)\equiv q([\fRb])\mod{4},$$
    for some quadratic function $q:H_2(\Xb;\ZZ/2)\to\ZZ/4$. From the previous \mbox{computations,} this yields $$2ab\equiv q([\cA(F)])\mod{4}\quad\text{and}\quad2\equiv q([\fRb])\mod{4}.$$
    Since $a$ and $b$ are both odd, this means that $2ab\not\equiv0\mod{4}$. As such, we cannot have $q([\cA(F)])=0$ and $q([\fRb])=0$, thus implying $[\cA(F)]\neq 0$ and $[\fRb]\neq0$.
\end{proof}

Let $Y$ denote the 2-fold covering of $\Xb$ ramified along $\cX(F)$. By \cref[Proposition]{prop:torsionHomology}, $H_1(Y)$ is torsion, and computations of $\chi(Y)=2ab+a+b+8$ and $\sigma(Y)=\linebreak-ab-a-b-5$ yield $$b_2^+(Y)=\frac{ab+1}{2}.$$
This implies the following result.

\begin{theorem}\label{thm:hyperboloid}
    Let $F$ be a $\fQ$-flexible curve in the hyperboloid with bidegree $(a,b)$ where both $a$ and $b$ are odd. Let $\ell^\pm$ and $\ell^0$ denote the number of ovals of the curve that bound from the outside a subset with positive, negative or zero Euler characteristic, respectively. Then: $$\ell^-+\ell^0\leqslant\frac{ab+1}{2}.$$
\end{theorem}

\begin{proof}
    We denote as $\Theta:Y\to\Xb$ the double branched cover of $(\Xb,\cX(F))$. For any oval $o\subset\RR F$, $\fR\smallsetminus\RR F$ has exactly two path-connected components that have $o$\linebreak
    as a part of their boundary. One is a punctured disc, and the other is a punctured torus. We denote as $C(o)\subset\cR$ the image under $p:X\to\bS^4$ of the punctured disc component, and as $C_\pm(o)\subset\fQb_\pm$ the pre-images under $\tilde{p}:\Xb\to\bS^4$ of $C(o)$. We set $\tilde{C}(o)=\Theta^{-1}(C_-(o))$. For an analogue of $\tilde{D}(J)$, there is a subtlety. Indeed, in $\RR F$, there may be several parallel copies of an $(\alpha,\beta)$-curve in $H_1(\fR;\ZZ)$, where $\alpha$ and $\beta$ are both odd and coprime. Each of these curves will lift in $\fQb$ to two copies.\linebreak
    If $\RR F$ contains ovals, then it is possible to choose one connected component $D_-$ of $\fQb\smallsetminus\RRb F$ that has one of those curves as a boundary component, and at least one oval as another boundary component, and which is included in $\fQb_-$. Define $\tilde{D}=\Theta^{-1}(D_-)$. By computations analogous to the $\CP^2$ case, we have the following:
    \begin{enumerate}[leftmargin=21pt,itemsep=3pt]
        \item If $o\subset\RR F$ is an oval, then $e(Y,\tilde{C}(o))=-4\chi(o)$, where $\chi(o)=\chi(C_-(o))$. In particular, $e(Y,\tilde{C}(o))\leqslant0$ if and only if $o$ is a non-empty oval.
        \item $Q_Y(\tilde{D},\tilde{D})\leqslant0$.
        \item If $o\neq o'$ are two distinct ovals, then $Q_Y(\tilde{C}(o),\tilde{C}(o'))=0$ and $Q_Y(\tilde{C}(o),\tilde{D})=0$.
    \end{enumerate}
    We can now apply \cref[Lemma]{lem:VZ-Smith} to the family composed of the collection of the $\tilde{C}(o)$ and of $\tilde{D}$.
\end{proof}
    \subsection*{Curves on an ellipsoid}

We now consider the other anti-holomorphic involution $c_2:([x_0:x_1],|y_0:y_1])\mapsto([\bar{y_0}:\bar{y_1}],[\bar{x_0}:\bar{x_1}])$ on $X=\CP^1\times\CP^1$. This time, we have $$\fR=\mathrm{Fix}(c_2)=\{(x,\bar{x})\mid x\in\CP^1\}\cong\bS^2,$$
and $X/c_2\cong\CPb^2$. Algebraic curves in $(X,c_2)$ necessarily have a bidegree of the form $(m,m)$ for some $m\geqslant1$. Consider a purely imaginary bidegree $(2,2)$ algebraic curve $\fQ$, and define flexible curves and $\fQ$-flexible curves as before. Note that we still keep the same basis for $H_2(X;\ZZ)$ as in the case of the hyperboloid.

\begin{theorem}\label{thm:ellipsoid}
    Let $F$ be a bidegree $(m,m)$ $\fQ$-flexible curve on the ellipsoid, with $m$ odd. Let $\ell^\pm$ and $\ell^0$ denote the number of connected components of $\fR\smallsetminus\RR F$ with positive, negative or zero Euler characteristic. Then $$\ell^0+\ell^-\leqslant\frac{m^2+1}{2}.$$
\end{theorem}

We have $$\vspace{.2em}e(X,\fR)=-2\text{ and }e(X,\fQ)=8,\vspace{.2em}$$
because $[\fR]=(\pm1,\mp1)\in H_2(X;\ZZ)$ (depending on a choice of orientation) and $[\fQ]=(2,2)$. Denoting the branched cover as $p:X\to\CPb^2$, we see that, letting $\cR=p(\fR)$ and $\cQ=p(\fQ)$, $$e(\CPb^2,\cR)=-4\quad\text{and}\quad e(\CPb^2,\cQ)=4.$$
In particular, $\cQ$ is a null-homologous Klein bottle in $H_2(\CPb^2;\ZZ/2)$, because it has even normal Euler number. This means that there is a well-defined 2-fold branched cover $\tilde{p}:\Xb\to\CPb^2$ ramified along $\cQ$. We let $\fRb=\tilde{p}^{-1}(\cQ)$ and $\fQb=\tilde{p}^{-1}(\cR)$, so that $$e(\Xb,\fRb)=2\text{ and }e(\Xb,\fQb)=-8.$$
A direct computation provides $$\chi(\Xb)=6\quad\text{and}\quad\sigma(\Xb)=-4,$$
with $H_1(\Xb)$ torsion. This is evidence to think that $\Xb\cong4\CPb^2$. What will be\linebreak
useful is knowing that $\Xb$ is negative definite, and so has intersection form $-I_4$ by Donaldson's theorem.

This time, the restriction $\tilde{p}:\fQb\to\cR$ is a two-fold covering of the 2-sphere, and is necessarily trivial. We set $\fQb=Q_1\sqcup Q_2$. Let $\tau:\Xb\to\Xb$ be the involution of\linebreak
$\Xb$ spanning $\Aut(\tilde{p})$. Denote as $R_1$ and $R_2$ the two subsets of $\cR\smallsetminus p(\RR F)$ with\linebreak
$\partial R_i=p(\RR F)$, and define $$\fQb_+=Q_1\cap\tilde{p}^{-1}(R_1)\sqcup Q_2\cap\tau(\tilde{p}^{-1}(R_1)),$$
$$\fQb_-=Q_2\cap\tilde{p}^{-1}(R_1)\sqcup Q_1\cap\tau(\tilde{p}^{-1}(R_1)).$$
We refer to \cref[Figure]{fig:two-spheres} for a representation. Of course, this definition depends on the choices of the labeling $Q_i$ of the two components of $\fQb$, as well as the choice of\linebreak
the labeling of the $R_i$. But ultimately, the inequality we obtain will not depend on these choices.

\begin{figure}[t]
    \centering
    \begin{tikzpicture}
        \node[inner sep=0pt] at (0,0) {\includegraphics{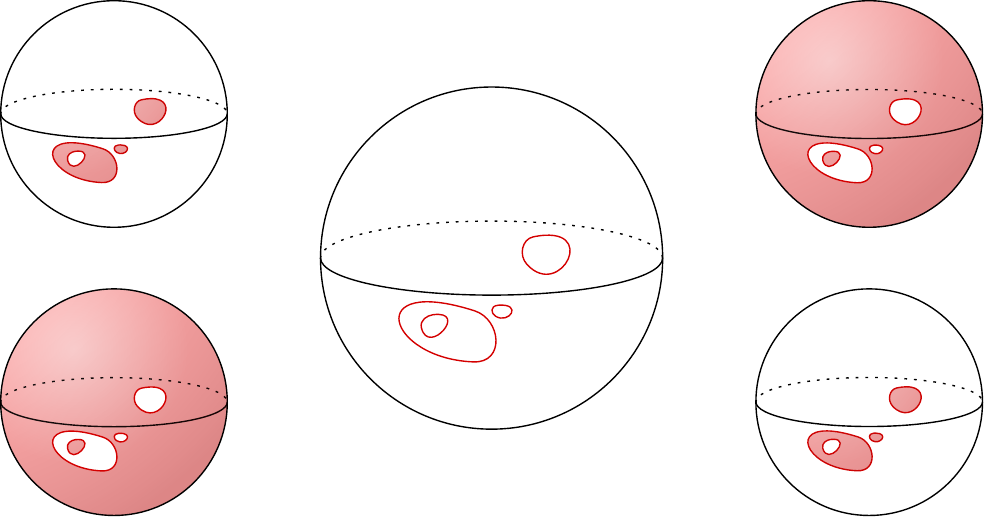}};
        \node at (2.5,0) {$\longleftarrow$};
        \node[above] at (2.5,0) {$\tilde{p}$};
        \node at (-2.4,0) {$\longrightarrow$};
        \node[above] at (-2.4,0) {$\tilde{p}$};
        \node at (-3.9,-3) {$\fQb_+$};
        \node at (3.9,-3) {$\fQb_-$};
        \node at (0,-2.2) {$p(\RR F)\subset\cR$};
    \end{tikzpicture}
    \caption{The two possible subsets $\fQb_\pm$, shaded. It is understood that the two spheres in the first row are $Q_1$, and the two in the second are $Q_2$.}
    \label{fig:two-spheres}
\end{figure}
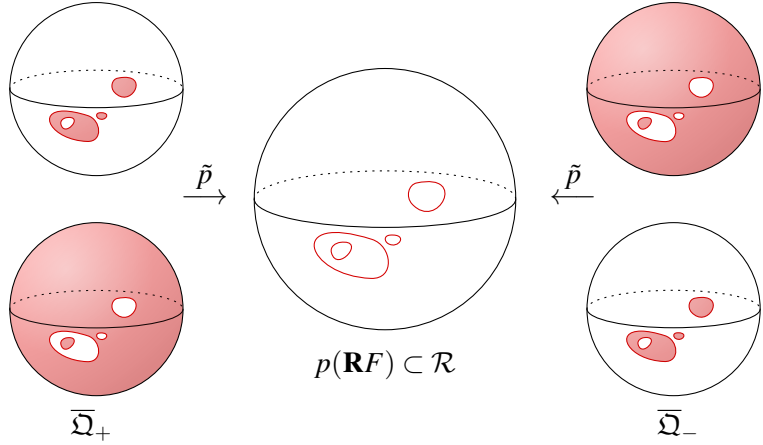

This allows for a definition of $\cA(F)$ such that $e(\fQb_+)=\frac{1}{2}e(\fQb)$ and $\chi(\fQb_+)=\frac{1}{2}\chi(\fQb)$. We obtain $$\chi(\cA(F))=-2m^2+2m+2\quad\text{and}\quad e(\cA(F))=2m^2-4.$$

Another key difference from the cases of $\CP^2$ and of the hyperboloid is that the\linebreak
second homology $H_2(\Xb;\ZZ/2)$ now has rank four (the intersection form of $\Xb$ is $-I_4$). To show that $\cA(F)$ and $\fRb$ are homologous mod 2 and not null-homologous, we need to eliminate more cases. We consider a basis of $H_2(\Xb;\ZZ)$ that diagonalizes the intersection form of $\Xb$. It also descends to a basis of $H_2(\Xb;\ZZ/2)$. If $(a,b,c,d)\in H_2(\Xb;\ZZ/2)$ denotes the homology class of $\cA(F)$ or $\fRb$, with $a,b,c,d\in\{0,1\}$, then the fact that $e(\Xb,\cA(F))$ and $e(\Xb,\fRb)$ are even implies that $a+b+c+d\equiv0\mod{2}$. There are 8 remaining cases: $(0,0,0,0)$, $(1,1,1,1)$, and the six cases of the type $(1,0,1,0)$ with two non-zero coefficients. \cref[Theorem]{thm:YamadaWhitney} rules out the zero homology class, as well as the $(1,1,1,1)$ one. Let $\cX(F)$ be $\cA(F)\cup\fRb$ with all $2m$ singularities removed accordingly to \cref[Lemma]{lem:res-sing}. This gives $$\chi(\cX(F))=-2m^2-2m+2\quad\text{and}\quad e(\Xb,\cX(F))=2m^2+4m-2.$$
One last application of \cref[Theorem]{thm:YamadaWhitney} provides $q([\cX(F)])\equiv0\mod{4}$. In particular, if, without loss of generality, we have $[\fRb]=(1,1,0,0)$, then this means that there are only two choices: $$[\cA(F)]=(1,1,0,0)\text{ or }(0,0,1,1).$$
That is, either $\cX(F)$ is null-homologous, in which case $[\fRb]=[\cA(F)]$, or it is a characteristic surface if $[\fRb]\neq[\cA(F)]$. Assuming that $\cX(F)$ is characteristic, the Guillou--Marin congruence (\cref[Theorem]{thm:GM}) applies and gives $$\beta(\Xb,\cX(F))\equiv-(m+1)^2\equiv0\text{ or }4\mod{8},$$
by inspection of the squares of odd integers mod $8$. Because the surface $\cX(F)$ has high genus, this method will a priori not yield any contradiction.

\begin{figure}[t]
        \centering
        \begin{tikzpicture}
            \node[inner sep=0pt] at (0,0) {\includegraphics[width=.63\textwidth]{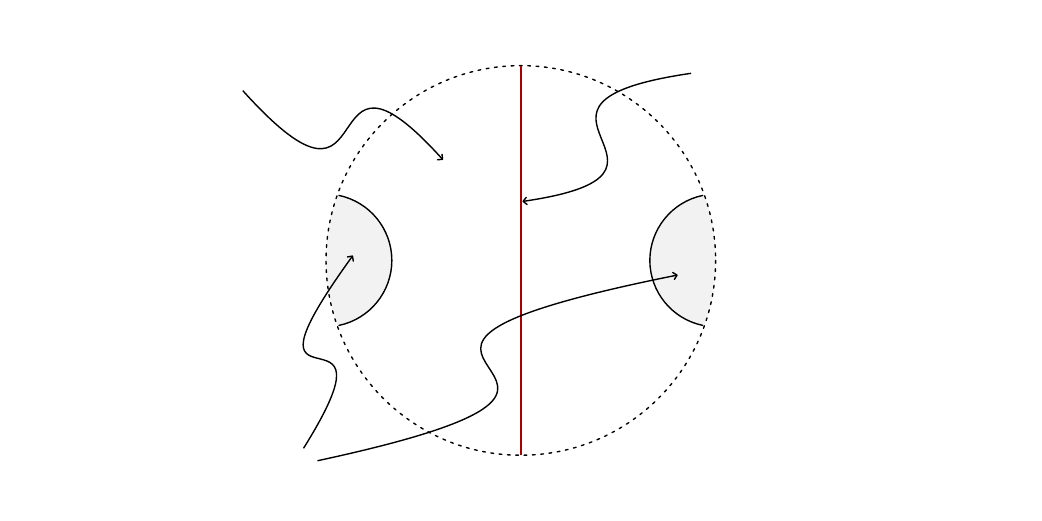}};
            \node at (1.7,1.6) {$\ell_i$};
            \node at (-2.8,1.6) {$R_i\smallsetminus D_i$};
            \node at (-1.85,-1.75) {$D_i$};
        \end{tikzpicture}
        \caption{The core of a Möbius strip can be seen as a real line in the associated real projective plane.\vspace{-1em}}
        \label{fig:coreMobius}
    \end{figure}
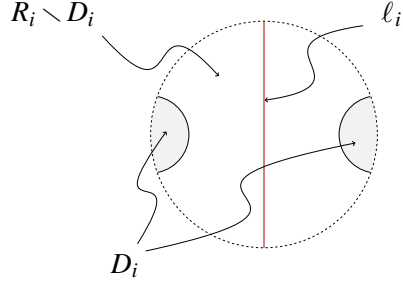

\vspace{-.2em}\begin{proposition}
    The surface $\cA(F)$ is homologous to $\fRb$ mod $2$.\vspace{-.2em}
\end{proposition}

\begin{proof}
    We start by describing the generators of the homology $H_2(\Xb;\ZZ/2)$. Consider a complex line $\CP^1\subset\CPb^2$ such that $\CP^1\cap\cQ=\varnothing$. This means that $\CP^1$ lifts to two spheres $S_1$ and $S_2$ in $\Xb$, each with $e(\Xb,S_i)=-1$. Moreover, because they are\linebreak
    disjoint, we have $Q_{\Xb,\ZZ/2}(S_1,S_2)\equiv0\mod{2}$, i.e. they are linearly independent. There are two more generators that come from the following construction.
    
    $\cQ$ is a Klein bottle, which can be seen as the desingularization of two real projective planes $R_1$ and $R_2$, with $R_1\tcap R_2=\{\ast\}$ and $e(\CPb^2,R_i)=1$. By \cref[Lemma]{lem:res-sing}, we see that this is possible from the computation $e(\CPb^2,\cQ)=4$. Let $x\in R_1\tcap R_2$ be the transverse intersection, and let $D_i\subset R_i$ be a small disc centered at $x$. $R_i\smallsetminus D_i$ is a\linebreak
    Möbius strip, whose core $\ell_i$ can be seen as a real projective line in $R_i$ (see \cref[Figure]{fig:coreMobius}).

    This real line separates a complex line $L_i$ into two components $L_i^\pm$ with $\partial L_i^\pm=\ell_i$. Let $\Sigma_i=\tilde{p}^{-1}(L_i^+)$. From $L_i^+\tcap\cQ=\ell_i=\partial L_i^+$, we see that $\Sigma_i$ is a sphere in $\Xb$ with $e(\Xb,\Sigma_i)=-1$. Moreover, $Q_{\ZZ/2}(\Sigma_1,\Sigma_2)\equiv0\mod{2}$ because $\Sigma_1\cap\Sigma_2=\varnothing$. Finally, because $L_i^\pm\cap\CP^1\subset\CPb^2\smallsetminus\cQ$, we have $\Sigma_i\tcap S_J$ is an even number of points, and thus $Q_{\ZZ/2}(\Sigma_i,S_j)\equiv0\mod{2}$. As such, $(S_1,S_2,\Sigma_1,\Sigma_2)$ is a basis for the homology $H_2(\Xb;\ZZ/2)$. From \cite[Lemma 3.4]{Nag00}, the surface $\tilde{p}^{-1}(\CP^1)\cup\fRb$ is mod $2$ characteristic in $\Xb$, and as such, we have $[\fRb]_{\ZZ/2}=[S_1]+[S_2]$. In order to show that $\cA(F)$ and $\fRb$ are homologous mod $2$, it suffices to prove that $Q_{\Xb,\ZZ/2}(\cA(F),\Sigma_i)\equiv\linebreak0\mod{2}$ for $i=1,2$. Equivalently, we need to show that $\cA(F)\tcap\Sigma_i$ is an \textit{even} number of points for $i=1,2$. Intersection points in $\cA(F)\tcap\Sigma_i$ come in two types:
    \begin{enumerate}[leftmargin=21pt,itemsep=3pt]
        \item intersections between $\tilde{p}^{-1}(F^+)$ and $\Sigma_i$; this number equals that of intersection points between $F^+$ and $\CP^1$, which is itself even because $e(\CPb^2,F^+)=2m^2$;
        \item intersections between $\fQb_+$ and $\Sigma_i$, itself also equal to $\#\cR\tcap\CP^1$, which is even because $e(\CPb^2,\cR)=-4$.\qedhere
    \end{enumerate}
\end{proof}

To prove \cref[Theorem]{thm:ellipsoid}, we apply the same method as before. Given a connected component $U\subset\fR\smallsetminus\RR F$ (or equivalently, $U\subset\cR\smallsetminus p(\RR F)$), the lift $\tilde{p}^{-1}(U)$ is two disjoint copies of $U$. We let $C_\pm(U)$ denote those copies, with the condition that $C_\pm(U)\subset\fQb_\pm$. As before, let $\tilde{C}(U)=\Theta^{-1}(C_-(U))$, with $\Theta:Y\to\Xb$ the double branched cover of $\Xb$ ramified over $\cX(F)$. We have $$e(Y,\tilde{C}(U))=-4\chi(U).$$
If $Q_Y$ is the intersection form of $Y$, and if $U$ and $V$ are two distinct components of $\fR\smallsetminus\RR F$, then there are two possibilities:
\begin{enumerate}[leftmargin=21pt,itemsep=3pt]
    \item $U\cap V=\varnothing$, in which case $C_-(U)\cap C_-(V)=\varnothing$, and thus $Q_Y(\tilde{C}(U),\tilde{C}(V))=0$.
    \item $U\cap V$ is a component of $\RR F$, in which case $C_-(U)\subset Q_i$ and $C_-(V)\subset Q_j$,\linebreak
    with $\{i,j\}=\{1,2\}$. In particular, we still have $C_-(U)\cap C_-(V)=\varnothing$.
\end{enumerate}

Finally, one applies the same arguments as before to the family $\{\tilde{C}(U)\}_{\chi(U)\leqslant0}$\linebreak
to obtain the claimed bound.

\section{Further comments}\label{sec:further}
    \subsection*{Other ways to resolve the singularities}

In order to take the $2$-fold branched cover, we added $\RPb^2$ to $\cA(F)$. This led us to resolve the $m$ singularities that arose. As suggested by Zvonilov in a personal communication, one could be tempted to use blow-ups and see what effect this has. But in order to ensure that the new surface $\cX(F)$ is still connected, we cannot blow-up all $m$ singularities. Doing this procedure to $m-1$ of those, and gluing a Hopf band for the last as we did previously, leads\linebreak{}
to the very same bound. That is, the 4-manifold $Y$ which is the double branched cover of $(m\CPb^2,\cX(F))$ still has $b_2^+(Y)=(m-1)^2/4$.
    \subsection*{Comparisons of our inequality}

Given a prime number $p$ and an integer $m\in\NN^\star$, we denote as $\nu_p(m)=\max\{n\in\NN\mid p^n\text{ divides }m\}$ the $p$-adic valuation of $m$. Define the function $h:\NN^\star\to\NN$ by $$h(m)=\max_{p\text{ prime}}p^{\nu_p(m)}.$$
That is, $h(m)$ is the largest prime power that divides $m$. Viro and Zvonilov's\linebreak
inequality, which holds for flexible curves, is\vspace{-.2em}

$$\ell^0+\ell^-\leqslant\frac{(m-3)^2}{4}+\frac{m^2-h(m)^2}{4h(m)^2}.\vspace{.5em}$$
We denote as $VZ(m)$ and $S(m)$ the bounds obtained by Viro and Zvonilov and\linebreak
ours, respectively. That is,\vspace{-.2em}

$$VZ(m)=\frac{(m-3)^2}{4}+\frac{m^2-h(m)^2}{4h(m)^2}\quad\text{and}\quad S(m)=\frac{(m-1)^2}{4}.\vspace{.5em}$$
For infinitely many degrees $m$, one has $S(m)<VZ(m)$. But in infinitely many\linebreak
others (e.g. when $m$ is a prime power), the converse holds. However, both are far\linebreak
from sharp estimates that can be obtained from considerations for \textit{algebraic} curves that come from Bézout theorem computations. That is, there are degrees $m$ for which $VZ(m)$ and $S(m)$ are both not realized as upper bounds for $\ell^0+\ell^-$. For instance, Zvonilov \citeyear{Zvo79} has the sharper estimate, valid for pseudoholomorphic curves, $$\ell^0+\ell^-\leqslant\frac{(m-1)(m-3)}{4}.$$

If one starts with Viro and Zvonilov's inequality in the case where $m+2$ is\linebreak
a prime power and the curve is $Q$-flexible of degree $m$, then one can perturb its union\linebreak
with the conic $Q$ into a non-singular degree $m+2$ flexible curve (which will have the same real set, for $\RR Q=\varnothing$), and obtain\vspace{-.2em}

$$\ell^0+\ell^-\leqslant\frac{(m-1)^2}{4}+\frac{(m+2)^2-h(m+2)^2}{4h(m+2)^2}=\frac{(m-1)^2}{4}.\vspace{1em}$$
That is, one can derive our \cref[Theorem]{thm:main} from Viro and Zvonilov's when $m+2$ is a prime power. On a side note, if the famous twin prime conjecture happens to be true, this means that there are infinitely many degrees $m$ for which $VZ(m)<S(m)$ and the bound $S(m)$ is a corollary of their bound.

By some easy number-theoretic considerations, one can show that there are\linebreak
infinitely many odd degrees $m$ such that neither of $m$ and $m+2$ are prime \mbox{powers,} and for which $S(m)<VZ(m)$. Indeed, the difference of the upper bounds in both inequalities is

$$VZ(m)-S(m)=\frac{1}{4}\left(\left[\frac{m}{h(m)}\right]^2-4m+7\right).\vspace{.5em}$$
With $m_p=1287\times429^{12p+1}$, one has $5|m_p+2$ and $7|m_p+2$, and $h(m_p)\in o(m_p^{19/40})$. In particular, the difference diverges to $+\infty$ on the degrees $m_p$.

The same conclusion can be derived when comparing the inequalities of The\-orems \ref{thm:hyperboloid} and \ref{thm:ellipsoid} with Zvonilov's work \citeyear{Zvo22}. It turns out that, for a curve of bidegree $(a,b)$ with $a$ and $b$ coprime, Zvonilov has no possibility to take a cyclic covering, and there is no inequality in those cases.
    \subsection*{Non-orientable flexible curves}

There is a new object that could be interesting to study: nonorientable flexible curve. The motivation comes from the observation that in the operation of taking double branched covers, orientability of the ramification locus is disregarded. This is not the case for other cyclic branched covers (and the methods from \cite{VZ92} cannot apply to non-orientable surfaces). We propose the following non-orientable analogue of \cref[Definition]{def:flexible}.

\begin{definition}
    Let $F\subset\CP^2$ be a closed, connected and non-orientable surface. We call $F$ a \textit{non-orientable degree $m$ and genus $h$ flexible curve} if the following conditions hold:
    \begin{enumerate}[label=(\roman*),leftmargin=21pt,itemsep=3pt]
        \item $\chi(F)=2-h$.
        \item $\conj(F)=F$.
        \item $e(\CP^2,F)=m^2$.
        \item For any $x\in\RR F=F\cap\RP^2$, we have $T_xF=T_x\RR F\oplus{\bf i}\cdot T_x\RR F$.
    \end{enumerate}
\end{definition}

\vspace{-.2em}What plays the role of asking that the integral homology class of $F$ is $m$ times a generator $[\CP^1]$ in $H_2(\CP^2;\ZZ)$ is the condition $e(\CP^2,F)=m^2$. In the traditional orientable case, we also had the condition that $\chi(F)=-m^2+3m$. This was a requirement of extremality in the genus bound proved by Kronheimer and Mrowka.\vspace{-.5em}

\begin{theorem}[{Thom conjecture, \cite{KM94}}]\label{thm:KM}
    Let $F\subset\CP^2$ be a smoothly embedded oriented and connected surface with $[F]=m[\CP^1]$. Then $\chi(F)\leqslant-m^2+3m$.\vspace{-.2em}
\end{theorem}

One could ask whether the implication $$e(\CP^2,F)=m^2\implies\chi(F)\leq-m^2+3m$$
holds for closed, connected nonorientable surfaces $F$ smoothly embedded in $\CP^2$. In fact, self-intersection numbers of non-orientable surfaces need not be squares. Given any $m\in\ZZ$, set $\Sigma(m)$ to be the collection of all smoothly embedded, \mbox{connected} and non-orientable surfaces $F\subset\CP^2$ with $e(\CP^2,F)=m$. We can define the\linebreak
following \textit{non-orientable genus function} of $\CP^2$: $$\tilde{g}:\ZZ\to\ZZ_{\leqslant1},\hspace{3.2em}$$
$$\hspace{1em}m\mapsto\max_{F\in\Sigma(m)}\chi(F).$$

We will later prove the following result.

\begin{theorem}\label{thm:NOThom}
    Here, $k\in\NN^\star$ denotes a non-negative integer.
    \begin{itemize}[leftmargin=21pt,itemsep=3pt]
        \item[\textup{(1)}] We have $\tilde{g}(0)=0$.
        \item[\textup{(2)}] Let $\ell\in\{0,1\}$ have the same parity as $k$. Then, on negative integers, we have $$\tilde{g}(-k)=2-\frac{k+\ell}{2}.$$
        \item[\textup{(3)}] On even positive integers, we have $$\tilde{g}(4k)=4-2k\quad\text{(for $k\geqslant2$)\quad}\text{and}\quad\tilde{g}(4k+2)=3-2k.$$
        We also have the special values $\tilde{g}(2)=1$ and $\tilde{g}(4)=0$.
        \item[\textup{(4)}] On odd positive integers, we have the bounds $$\tilde{g}(4k+1)\geqslant2-2k\quad\text{and}\quad\tilde{g}(4k+3)\geqslant1-2k.$$
        We also have the special values $\tilde{g}(1)=0$, $\tilde{g}(3)=1$, $\tilde{g}(5)=0$, $\tilde{g}(7)=-1$ and $\tilde{g}(9)=-2$.
    \end{itemize}
\end{theorem}

In the previous theorem, one can now look at the values of $\tilde{g}(m^2)$. We obtain
\begingroup
\setlength\lineskiplimit{\maxdimen}
$$\raisebox{.45em}{$\biggl\{$}\begin{matrix}\tilde{g}(m^2)=\frac{8-m^2}{2}&\text{if $m$ is even,}\\\vspace{1pc}\tilde{g}(m^2)\geqslant\frac{5-m^2}{2}&\text{if $m$ is odd.}\end{matrix}$$
\endgroup

In particular, we see that the non-orientable analogue $\tilde{g}(m^2)\leqslant-m^2+3m$ of\linebreak
\cref[Theorem]{thm:KM} has the quadratic term off by 50\%. Non-orientable flexible curves still\linebreak
share some properties with traditional flexible curves. More precisely, we show the following.

\begin{proposition}\label{prop:NOproperties}
    Let $F\subset\CP^2$ be a non-orientable flexible curve of degree $m$. Then
    \begin{itemize}[leftmargin=21pt,itemsep=3pt]
        \item[\textup{(1)}] $\chi(F)$ is an even integer;
        \item[\textup{(2)}] $\RR F$ realizes the non-trivial homology class in $H_1(\RP^2;\ZZ)$ if and only if $m$ is odd, and it has exactly one pseudo-line in this case;
        \item[\textup{(3)}] $F$ satisfies the Harnack bound: $b_0(\RR F)\leqslant3-\chi(F)$.
    \end{itemize}
\end{proposition}

\begin{proof}
    The first claim is a consequence of the following result, which is a generalization of the well-known Whitney congruence.
    
    \begin{theorem}[{\cite[Theorems 1.2 and 1.4]{Yam95}}]\label{thm:YamadaWhitney}
        Let $X$ be a closed, connected, oriented 4-manifold.
        \begin{itemize}[leftmargin=21pt,itemsep=3pt]
            \item[\textup{(1)}] If $H_1(X;\ZZ)=0$, define $q:H_2(X;\ZZ/2)\to\ZZ/4$ by setting, for $\xi\in H_2(X;\ZZ/2)$, $q(\xi)=Q_X(\tilde{\xi},\tilde{\xi})\mod{4}$, where $\tilde{\xi}\in H_2(X;\ZZ)$ is any integral lift of $\xi$. Then, for
            any embedded, closed, connected (not necessarily orientable) surface $F\subset X$, $$e(X,F)+2\chi(F)\equiv q([F])\mod{4}.$$
            \item[\textup{(2)}] Without the assumption that $H_1(X;\ZZ)=0$, the map $$q:H_2(X;\ZZ/2)\to\ZZ/4$$
            defined by $q([F])=e(X,F)+2\chi(F)$ is a well-defined $\ZZ/4$-quadratic map.
        \end{itemize}
    \end{theorem}
    
    Indeed, if $e(\CP^2,F)=m^2$, then one inspects two cases, depending on the parity of $m$. If $m$ is even, then $[F]=0\in H_2(\CP^2;\ZZ/2)$, and if $m$ is odd, then $[F]$ is the generator of $H_2(\CP^2;\ZZ/2)$. Both cases yield $\chi(F)\equiv0\mod{2}$.

    For the other claims, the classical proofs for flexible curves, found, for instance in Viro's lecture notes, work word for word.
\end{proof}

We call a nonorientable flexible curve of degree $m$ \textit{$Q$-flexible} if, as before, the intersection $F\tcap Q$ is $2m$ points. Then, we have the following result.

\begin{theorem}\label{thm:NOVZ}
    Let $F\subset\CP^2$ be a non-orientable $Q$-flexible of odd degree $m$. Then $$\ell^0+\ell^-\leqslant-\frac{\chi(F)}{2}-\frac{m^2-1}{4}+m.$$
\end{theorem}

\begin{proof}
    The only difference with traditional flexible curves is that one needs to do all\linebreak
    the computations in terms of $\chi(F)$. Indeed, starting at the level of $\bS^4$, the surface $F/\conj$ needs not be orientable anymore. One checks that, for the Arnold surface, we have $$\chi(\cA(F))=\chi(F)-m+1\quad\text{and}\quad e(\CPb^2,\cA(F))=m^2-2,$$
    and for the smoothing $\cX(F)$, we obtain $$\chi(\cX(F))=\chi(F)-3m+2\quad\text{and}\quad e(\CPb^2,\cX(F))=m^2+2m-1.$$
    This gives, denoting as $Y$ the double branched cover of $(\CPb^2,\cX(F))$, $$b_2^+(Y)=-\frac{\chi(F)}{2}-\frac{m^2-1}{4}+m.$$
    Note that $\chi(F)\leqslant0$ is necessarily even, as seen in \cref[Proposition]{prop:NOproperties}.
\end{proof}

In regards to \cref[Theorem]{thm:NOThom}, we conjecture the following.

\begin{conjecture}
    The lower bounds for $\tilde{g}$ over non-negative odd integers are equali\-ties.
\end{conjecture}

If this holds, then one may add to the definition of a nonorientable flexible\linebreak
curve $F$ of degree $m$ that they must satisfy the extremal bound $\chi(F)=\tilde{g}(m)$. In this case, \cref[Theorem]{thm:NOVZ} becomes $$\ell^0+\ell^-\leqslant m-1,$$
whereas the Harnack bound gives $$b_0(\RR F)\leqslant\frac{m^2+1}{2}.$$
This is to be compared to $$b_0(\RR F)\leqslant\frac{m^2-3m+4}{2}\sim\frac{m^2}{2}\quad\text{and}\quad\ell^0+\ell^-\leqslant\frac{m^2-2m+1}{4}\sim\frac{m^2}{4}$$
for traditional $Q$-flexible curves.
    \subsection*{\texorpdfstring{Proof of \cref[Theorem]{thm:NOThom}}{Proof of Theorem 5.3}}

The two steps of the proof are to
\begin{enumerate}[leftmargin=21pt,itemsep=3pt]
    \item obtain upper bounds for $\chi(F)$ given $e(\CP^2,F)=m$, and
    \item construct a surface realizing that upper bound.
\end{enumerate}

To this end, we will use the following.

\begin{theorem}[{\citeLRS}]
    Let $X$ be a closed, connected, oriented, positive definite 4-manifold with $H_1(X;\ZZ)=0$, and let $F\subset X$ be a closed, connected, nonorientable surface with nonorientable genus $h(F)=2-\chi(F)$.\linebreak
    Denote as $\ell(F)$ the minimal self-intersection of an integral lift of $[F]\in H_2(X;\ZZ/2)$. Then $$e(X,F)\geq\ell(F)-2h(F).$$
\end{theorem}

This allows us to obtain the upper bounds\vspace{-.5em}

$$\tilde{g}(-k)\leqslant2-\frac{k+\ell}{2}\vspace{.5em}$$
for $k\in\NN^\star$ and $\ell\in\{0,1\}$ having the same parity as $k$. Indeed, if $F\subset\CP^2$ has $e(\CP^2,F)=-k$, then $\ell(F)=\ell$, because $[F]\neq0\in H_2(\CP^2;\ZZ/2)$ if and only\linebreak
if $-k$ is odd, in which case a complex line is an integral lift of $F$ with minimal self-intersection.

Another method (which worked for the orientable Thom conjecture in degree~$4$, for instance) is to make use of homological information of the double branched cover of $\CP^2$ ramified along $F$. More precisely, we have the following.

\begin{proposition}\label{prop:dBCObstruction}
    Let $F\subset\CP^2$ be a closed, connected surface such that $[F]=0\in H_2(\CP^2;\ZZ/2)$ (or equivalently, such that $e(\CP^2,F)$ is even). Then
    
    $$\chi(F)\leqslant4-\frac{e(\CP^2,F)}{2}.$$
\end{proposition}

\begin{proof}
    Let $Y$ denote the double branched cover of $\CP^2$ ramified over $F$. We\linebreak
    compute\vspace{-.5em}
    
    $$\chi(Y)=6-\chi(F)\quad\text{and}\quad\sigma(Y)=2-\frac{e(\CP^2,F)}{2}.\vspace{.5em}$$
    Moreover, by reasoning analogous to the proof of \cref[Proposition]{prop:torsionHomology}, we have $b_1(Y)=b_3(Y)=0$, so that $b_2(Y)=4-\chi(F)$. If one considers any orientable surface\linebreak
    $\Sigma\subset\CP^2$ which is not null-homologous in $H_2(\CP^2;\ZZ)$, and which is transverse to $F$, we see that $e(\CP^2,\Sigma)>0$, and $\Sigma$ lifts in $Y$ to a surface $\tilde{\Sigma}$ with $e(Y,\tilde{\Sigma})=2e(\CP^2,\Sigma)>0$.  This implies that\vspace{-.5em}
    
    $$b_2^+(Y)=\frac{b_2(Y)+\sigma(Y)}{2}\geqslant1,\vspace{.5em}$$
    yielding the inequality that was claimed.
\end{proof}

Note that unless $e(\CP^2,F)\geqslant8$, the previous result only gives trivialities, since $\chi(F)\leqslant1$ for a non-orientable surface. This is enough to obtain the upper bounds $$\tilde{g}(4k)\leqslant4-2k\quad\text{and}\quad\tilde{g}(4k+2)\leqslant3-2k$$
for $k\in\NN^\star$. Note that if $k=1$, then the bound $\tilde{g}(4)\leqslant2$ is vacuous.

What remains to do is
\begin{enumerate}[leftmargin=21pt,itemsep=3pt]
    \item compute the special values of $\tilde{g}$ at $0$, $1$, $2$, $3$, $4$, $5$, $7$ and $9$;
    \item construct surfaces $F\subset\CP^2$ that realize the upper bounds for $\tilde{g}(-k)$, $\tilde{g}(4k)$\linebreak
    and $\tilde{g}(4k+2)$;
    \item construct surfaces $F\subset\CP^2$ to derive lower bounds for $\tilde{g}(4k+1)$ and $\tilde{g}(4k+3)$.
\end{enumerate}

To obtain upper bounds for $\tilde{g}$ on odd integers $\leqslant9$, we will need the following.

\begin{theorem}[\cite{GM80}]\label{thm:GM}
    Let $F\subset X$ be a mod 2 characteristic surface in a closed, connected oriented 4-manifold with $H_1(X;\ZZ)=0$. Then $$\sigma(X)-e(X,F)\equiv2\beta(X,F)\mod{16},$$
    where $\beta(X,F)$ is the Brown invariant of the embedding.
\end{theorem}

We shall recall what $\beta(X,F)$ is. The Guillou--Marin form $$\varphi:H_1(F;\ZZ/2)\to\ZZ/4$$
is defined as follows. Because $H_1(X;\ZZ)=0$, any immersed circle $\msC\looparrowright F$ bounds an immersed orientable surface $\msD\looparrowright X$. Isotope $\msD$ (relatively to its boundary) such that it is transverse to $F$. The normal bundle $\nu\msD$ of $\msD$ in $X$ is trivial, and as such, so is $\nu\msD|_\msC$. Considering the normal bundle $\nu\msC$ of $\msC$ in $F$ as a subbundle $\nu\msC<\nu\msD|_\msC$, count the number $n(\msD)$ of right-handed half-twists with respect to the fixed trivialization $\nu\msD|_\msC\cong\msC\times\RR^2$. Define
\begin{equation}\label{eq:marin-form}
    \varphi(\msC)=n(\msD)+2\msD\cdot F+2e(F,\msC)\mod{4},\tag{$\ast$}
\end{equation}
where $\msD\cdot F$ is the number of transverse intersection points $F\tcap\msD$ taken mod 2. Then this definition does not depend on any of the choices made, and $\varphi(\msC)$ depends only on the homology class $[\msC]\in H_1(F;\ZZ/2)$. This defines a quadratic map $$\varphi:H_1(F;\ZZ/2)\to\ZZ/4,\pagebreak$$
to which one can regard its \textit{Brown invariant}

\begin{equation}\label{eq:brown-invariant}
    \beta(X,F)\overset{\text{def.}}{=}\left(\frac{1}{\sqrt{2}}\right)^{b_1(F;\ZZ/2)}\sum_{x\in H_1(F;\ZZ/2)}\sqrt{-1}^{\varphi(x)}.\tag{$\ast\ast$}\vspace{.5em}
\end{equation}

For instance, we can compute the possible values of $\beta(\CP^2,K)$ where $K\subset\CP^2$ is a Klein bottle. We refer to \cref[Figure]{fig:basis-Klein} for a choice of generators $a$ and $b$ of $H_1(K;\ZZ)=\{0,a,b,a+b\}$.

\begin{figure}[t]
    \centering
    \begin{tikzpicture}
        \node[inner sep=0pt] at (0,0) {\includegraphics{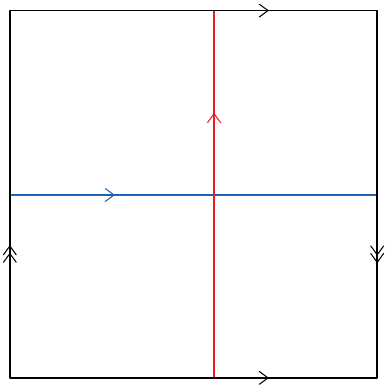}};
        \node at (.5,.78) {$b$};
        \node at (-.85,.31) {$a$};
    \end{tikzpicture}
    \begingroup
    \captionsetup{width=.9\linewidth}
    \caption{The standard basis for the first homology of the Klein bottle.}\label{fig:basis-Klein}
    \endgroup
\end{figure}

From $\varphi(a+b)=\varphi(a)+\varphi(b)+2a\cdot b=\varphi(a)+\varphi(b)+2$, it suffices to compute $\varphi(a)$ and $\varphi(b)$. One checks that $\varphi(a)\in\{1,3\}$ and $\varphi(b)\in\{0,2\}$, by computing each\linebreak
term in \ceqref{eq:marin-form}. Therefore, it suffices to inspect each of the four cases individually, and plug the values in \ceqref{eq:brown-invariant} to obtain $$\beta(\CP^2,K)\in\{0,2,6\}.$$

\begin{proposition}
    We have $\tilde{g}(1)\leqslant0$, $\tilde{g}(3)\leqslant1$, $\tilde{g}(5)\leqslant0$, $\tilde{g}(7)\leqslant-1$ and $\tilde{g}(9)\leqslant-2$.
\end{proposition}

\begin{proof}
    Note that \cref[Theorem]{thm:YamadaWhitney} gives
    $$e(\CP^2,F)\in\{1,5,9\}\implies\chi(F)\text{ is even},$$
    $$e(\CP^2,F)\in\{3,7\}\implies\chi(F)\text{ is odd}.$$
    In particular, from $\chi(F)\leqslant1$ always, we see that the only non-trivial bounds are for $\tilde{g}(7)$ and $\tilde{g}(9)$.\vspace{.5em}
    
    \noindent
    (1) Assume that $F\subset\CP^2$ has $e(\CP^2,F)=7$ and $\chi(F)=1$ (that is, $F$ is diffeomorphic to a projective plane). Then, the Guillou--Marin congruence (\cref[Theorem]{thm:GM}) gives $$1-7\equiv2\beta(\CP^2,F)\mod{16}.$$
    A simple calculation as before for the possible values of $\beta(\CP^2,K)$ gives that, in the case of the projective plane, $\beta(\CP^2,\RP^2)=1$, a contradiction.
    
    \noindent
    (2) If $F$ now has $e(\CP^2,F)=9$ and $\chi(F)=0$ (that is, $F$ is a Klein bottle), we\linebreak
    can use the Guillou--Marin congruence again to derive $$\beta(\CP^2,F)\equiv4\mod{8}.$$
    The previous calculations gave $\beta(\CP^2,K)\in\{0,2,6\}$, which is a contradiction.\qedhere
\end{proof}

Now, the only things that remain to do are to construct surfaces that realize the upper bounds obtained thus far, and compute upper bounds for $\tilde{g}$ in the special\linebreak
cases not covered yet. For constructions of surfaces, we will make use of \textit{local surfaces}. Recall the so-called Whitney--Massey theorem.

\begin{theorem}[\cite{Mas69}]
    Let $F\subset\bS^4$ be a closed connected non-orientable surface. Then $$e(\bS^4,F)\in\{2\chi(F)-4,2\chi(F),\dots,4-2\chi(F)\}.$$
    All tuples $(e,\chi)$ satisfying this condition are realizable by a closed, connected,\linebreak
    nonorientable surface.
\end{theorem}

A surface $F$ embedded in the $4$-ball $\BB^4$ that realizes an admissible tuple $(e,\chi)$ will be called a \textit{local surface}. Note that the previous theorem ensures that those always exist for any admissible tuple.

\begin{proposition}
    All upper bounds obtained for $\tilde{g}$ so far are sharp. We have $$\tilde{g}(4k+1)\geqslant2-2k\text{ and }\tilde{g}(4k+3)\geqslant1-2k.$$
\end{proposition}

\begin{proof}
    Fix $k\in\NN^\star$ a non-negative integer.\vspace{.5em}
    
    \noindent
    (1) If $k=2p$ is even, consider a local surface $F\subset\BB^4\subset\CP^2$ of genus $p$ and\linebreak
    self-intersection $-2p$. Then $\chi(F)=2-p=2-k/2$ and $e(\CP^2,F)=-k$. If now $k=2p+1$ is odd, choose $F\subset\BB^4$ a local surface of genus $p+1$ and self-intersection $-2(p+1)$. Embed the $4$-ball $\BB^4$ in $\CP^2$ away from a fixed complex line $L\subset\CP^2$, and consider the surface $F\#L$ obtained by connecting $F$ and $L$ with a small tube. Then, by noting that $L$ is a $2$-sphere with self-intersection $+1$, we have $$\chi(F\# L)=\chi(F)=2-(p+1)=2-\frac{k+1}{2}$$
    and $$e(\CP^2,F\#L)=-2(p+1)+1=-k.$$
    In both cases, this implies that $$\tilde{g}(-k)\geqslant2-\frac{k+\ell}{2}.\pagebreak$$
    
    \noindent
    (2) Assume that $k\geqslant2$. Let $F\subset\BB^4$ be a local surface of genus $2(k-1)$ and\linebreak
    self-intersection $4(k-1)$. Embed the $4$-ball away from the conic $Q$, and consider the surface $F\#Q$ (again, by connecting them with a small tube). We compute, using that $Q$ is a $2$-sphere with self-intersection $+4$, $$\chi(F\#Q)=4-2k\text{ and }e(\CP^2,F\#Q)=4k,$$
    yielding the lower bound $\tilde{g}(4k)\geqslant4-2k$. Note that in the case $k=1$, this works, but gives an orientable surface (the $2$-sphere $Q$).\vspace{.3em}
    
    \noindent
    (3) Let $F\subset\BB^4$ be a local surface of genus $2k-1$ and self-intersection $4k-2$.\linebreak
    Tubing with the conic $Q$ gives $$\chi(F\#Q)=3-2k\text{ and }e(\CP^2,F\#Q)=4k+2,$$
    which provides us with $\tilde{g}(4k+2)\geqslant3-2k$.\vspace{.3em}
    
    \noindent
    (4) One can do the same to derive the bounds $$\tilde{g}(4k+1)\geqslant2-2k\quad\text{and}\quad\tilde{g}(4k+3)\geqslant1-2k.$$
    Indeed, taking $F$ to be a local surface of genus $2k$ (\textit{resp.} $2k+1$) and self-intersection $4k$ (\textit{resp.} $4k+2$), this can be embedded away from a complex line $L$. Looking at the surface $F\#L$ gives the lower bounds. The special cases for $\tilde{g}(5),\dots,\tilde{g}(9)$ are covered by this construction. For $\tilde{g}(1)$, it suffices to consider a local Klein bottle with self-intersection $0$ and tubing it with a complex line, and for $\tilde{g}(3)$, the conic $Q$ can be tubed to $\RP^2=\mathrm{Fix}(\conj)$.\qedhere
\end{proof}

To conclude the proof of \cref[Theorem]{thm:NOThom}, there only remains to compute three\linebreak
special values that have not been covered yet:
\begin{enumerate}[leftmargin=21pt,itemsep=3pt]
    \item $\tilde{g}(0)=0$, because the Euler characteristic must be even, and a local Klein bottle with zero self-intersection gives a lower bound.
    \item $\tilde{g}(2)=1$, since a local projective plane with self-intersection $+2$ works.
    \item $\tilde{g}(4)=0$, as the Euler characteristic must be even, and a local Klein bottle\linebreak
    with self-intersection $+4$ suffices.
\end{enumerate}

\section*{Acknowledgements}

The author is very grateful to T. Fiedler for all the insight he gave and for all\linebreak
the discussions. The author would also like to thank M. Golla for his very nice\linebreak
explanations, and D. Moussard for all the precious time she gave. The author is also deeply thankful to the anonymous referee for helpful comments and suggestions. This work received support from the University Research School EUR-MINT (state support managed by the National Research Agency for Future Investments program bearing the reference ANR-18-EURE-0023).

\emergencystretch=1em % overfull hbox
\nocite{*}
%\printbibliography

\bibliography{References}

\begin{thebibliography}{}

\bibitem [\protect \citeauthoryear {%
Gilmer%
}{%
Gilmer%
}{%
{\protect \APACyear {1991}}%
}]{%
Gil91}
\APACinsertmetastar {%
Gil91}%
\begin{APACrefauthors}%
Gilmer, P.%
\end{APACrefauthors}%
\unskip\
\newblock
\APACrefYearMonthDay{1991}{}{}.
\newblock
{\BBOQ}\APACrefatitle {Algebraic curves in {${\bf R}{\rm P}(1)\times{\bf R}{\rm
  P}(1)$}} {Algebraic curves in {${\bf R}{\rm P}(1)\times{\bf R}{\rm
  P}(1)$}}.{\BBCQ}
\newblock
\APACjournalVolNumPages{Proc. Amer. Math. Soc.}{113}{1}{47--52}.
\PrintBackRefs{\CurrentBib}

\bibitem [\protect \citeauthoryear {%
Gilmer%
}{%
Gilmer%
}{%
{\protect \APACyear {1992}}%
}]{%
Gil92}
\APACinsertmetastar {%
Gil92}%
\begin{APACrefauthors}%
Gilmer, P.%
\end{APACrefauthors}%
\unskip\
\newblock
\APACrefYearMonthDay{1992}{}{}.
\newblock
{\BBOQ}\APACrefatitle {Real algebraic curves and link cobordism} {Real
  algebraic curves and link cobordism}.{\BBCQ}
\newblock
\APACjournalVolNumPages{Pacific J. Math.}{153}{1}{31--69}.
\PrintBackRefs{\CurrentBib}

\bibitem [\protect \citeauthoryear {%
Golla%
}{%
Golla%
}{%
{\protect \APACyear {2021}}%
}]{%
Gol21}
\APACinsertmetastar {%
Gol21}%
\begin{APACrefauthors}%
Golla, M.%
\end{APACrefauthors}%
\unskip\
\newblock
\APACrefYearMonthDay{2021}{}{}.
\newblock
\APACrefbtitle {Branched covers in low dimensions.} {Branched covers in low
  dimensions.}
\newblock
\APAChowpublished {University of Nantes, Spring term}.
\newblock
\APACrefnote{Lecture notes}
\PrintBackRefs{\CurrentBib}

\bibitem [\protect \citeauthoryear {%
Gompf%
\ \BBA {} Stipsicz%
}{%
Gompf%
\ \BBA {} Stipsicz%
}{%
{\protect \APACyear {1999}}%
}]{%
GS99}
\APACinsertmetastar {%
GS99}%
\begin{APACrefauthors}%
Gompf, R\BPBI E.%
\BCBT {}\ \BBA {} Stipsicz, A\BPBI I.%
\end{APACrefauthors}%
\unskip\
\newblock
\APACrefYear{1999}.
\newblock
\APACrefbtitle {{$4$}-manifolds and {K}irby calculus} {{$4$}-manifolds and
  {K}irby calculus}\ (\BVOL~20).
\newblock
\APACaddressPublisher{}{American Mathematical Society, Providence, RI}.
\PrintBackRefs{\CurrentBib}

\bibitem [\protect \citeauthoryear {%
Guillou%
\ \BBA {} Marin%
}{%
Guillou%
\ \BBA {} Marin%
}{%
{\protect \APACyear {1980}}%
}]{%
GM80}
\APACinsertmetastar {%
GM80}%
\begin{APACrefauthors}%
Guillou, L.%
\BCBT {}\ \BBA {} Marin, A.%
\end{APACrefauthors}%
\unskip\
\newblock
\APACrefYearMonthDay{1980}{}{}.
\newblock
{\BBOQ}\APACrefatitle {Une extension d'un th\'{e}or\`eme de {R}ohlin sur la
  signature} {Une extension d'un th\'{e}or\`eme de {R}ohlin sur la
  signature}.{\BBCQ}
\newblock
\BIn{} \APACrefbtitle {Seminar on {R}eal {A}lgebraic {G}eometry ({P}aris,
  1977/1978 and {P}aris, 1978/1979)} {Seminar on {R}eal {A}lgebraic {G}eometry
  ({P}aris, 1977/1978 and {P}aris, 1978/1979)}\ (\BVOL~9, \BPGS\ 69--80).
\newblock
\APACaddressPublisher{}{Univ. Paris VII, Paris}.
\PrintBackRefs{\CurrentBib}

\bibitem [\protect \citeauthoryear {%
Hirzebruch%
}{%
Hirzebruch%
}{%
{\protect \APACyear {1969}}%
}]{%
Hir69}
\APACinsertmetastar {%
Hir69}%
\begin{APACrefauthors}%
Hirzebruch, F.%
\end{APACrefauthors}%
\unskip\
\newblock
\APACrefYearMonthDay{1969}{}{}.
\newblock
{\BBOQ}\APACrefatitle {The signature of ramified coverings} {The signature of
  ramified coverings}.{\BBCQ}
\newblock
\BIn{} \APACrefbtitle {Global {A}nalysis ({P}apers in {H}onor of {K}.
  {K}odaira)} {Global {A}nalysis ({P}apers in {H}onor of {K}. {K}odaira)}\
  (\BPGS\ 253--265).
\newblock
\APACaddressPublisher{}{Univ. Tokyo Press, Tokyo}.
\PrintBackRefs{\CurrentBib}

\bibitem [\protect \citeauthoryear {%
Kronheimer%
\ \BBA {} Mrowka%
}{%
Kronheimer%
\ \BBA {} Mrowka%
}{%
{\protect \APACyear {1994}}%
}]{%
KM94}
\APACinsertmetastar {%
KM94}%
\begin{APACrefauthors}%
Kronheimer, P\BPBI B.%
\BCBT {}\ \BBA {} Mrowka, T\BPBI S.%
\end{APACrefauthors}%
\unskip\
\newblock
\APACrefYearMonthDay{1994}{}{}.
\newblock
{\BBOQ}\APACrefatitle {The genus of embedded surfaces in the projective plane}
  {The genus of embedded surfaces in the projective plane}.{\BBCQ}
\newblock
\APACjournalVolNumPages{Math. Res. Lett.}{1}{6}{797--808}.
\PrintBackRefs{\CurrentBib}

\bibitem [\protect \citeauthoryear {%
Kuiper%
}{%
Kuiper%
}{%
{\protect \APACyear {1974}}%
}]{%
Kui74}
\APACinsertmetastar {%
Kui74}%
\begin{APACrefauthors}%
Kuiper, N\BPBI H.%
\end{APACrefauthors}%
\unskip\
\newblock
\APACrefYearMonthDay{1974}{}{}.
\newblock
{\BBOQ}\APACrefatitle {The quotient space of {${\bf C}P(2)$} by complex
  conjugation is the {$4$}-sphere} {The quotient space of {${\bf C}P(2)$} by
  complex conjugation is the {$4$}-sphere}.{\BBCQ}
\newblock
\APACjournalVolNumPages{Math. Ann.}{208}{}{175--177}.
\PrintBackRefs{\CurrentBib}

\bibitem [\protect \citeauthoryear {%
Lee%
\ \BBA {} Weintraub%
}{%
Lee%
\ \BBA {} Weintraub%
}{%
{\protect \APACyear {1995}}%
}]{%
LW95}
\APACinsertmetastar {%
LW95}%
\begin{APACrefauthors}%
Lee, R.%
\BCBT {}\ \BBA {} Weintraub, S\BPBI H.%
\end{APACrefauthors}%
\unskip\
\newblock
\APACrefYearMonthDay{1995}{}{}.
\newblock
{\BBOQ}\APACrefatitle {On the homology of double branched covers} {On the
  homology of double branched covers}.{\BBCQ}
\newblock
\APACjournalVolNumPages{Proc. Amer. Math. Soc.}{123}{4}{1263--1266}.
\PrintBackRefs{\CurrentBib}

\bibitem [\protect \citeauthoryear {%
Letizia%
}{%
Letizia%
}{%
{\protect \APACyear {1984}}%
}]{%
Let84}
\APACinsertmetastar {%
Let84}%
\begin{APACrefauthors}%
Letizia, M.%
\end{APACrefauthors}%
\unskip\
\newblock
\APACrefYearMonthDay{1984}{}{}.
\newblock
{\BBOQ}\APACrefatitle {Quotients by complex conjugation of nonsingular quadrics
  and cubics in {${\bf P}^{3}_{{\bf C}}$} defined over {${\bf R}$}} {Quotients
  by complex conjugation of nonsingular quadrics and cubics in {${\bf
  P}^{3}_{{\bf C}}$} defined over {${\bf R}$}}.{\BBCQ}
\newblock
\APACjournalVolNumPages{Pacific J. Math.}{110}{2}{307--314}.
\PrintBackRefs{\CurrentBib}

\bibitem [\protect \citeauthoryear {%
Levine%
, Ruberman%
\BCBL {}\ \BBA {} Strle%
}{%
Levine%
\ \protect \BOthers {.}}{%
{\protect \APACyear {2015}}%
}]{%
LRS15}
\APACinsertmetastar {%
LRS15}%
\begin{APACrefauthors}%
Levine, A\BPBI S.%
, Ruberman, D.%
\BCBL {}\ \BBA {} Strle, S.%
\end{APACrefauthors}%
\unskip\
\newblock
\APACrefYearMonthDay{2015}{}{}.
\newblock
{\BBOQ}\APACrefatitle {Nonorientable surfaces in homology cobordisms}
  {Nonorientable surfaces in homology cobordisms}.{\BBCQ}
\newblock
\APACjournalVolNumPages{Geom. Topol.}{19}{1}{439--494}.
\newblock
\APACrefnote{With an appendix by Ira M. Gessel}
\PrintBackRefs{\CurrentBib}

\bibitem [\protect \citeauthoryear {%
Massey%
}{%
Massey%
}{%
{\protect \APACyear {1969}}%
}]{%
Mas69}
\APACinsertmetastar {%
Mas69}%
\begin{APACrefauthors}%
Massey, W\BPBI S.%
\end{APACrefauthors}%
\unskip\
\newblock
\APACrefYearMonthDay{1969}{}{}.
\newblock
{\BBOQ}\APACrefatitle {Proof of a conjecture of {W}hitney} {Proof of a
  conjecture of {W}hitney}.{\BBCQ}
\newblock
\APACjournalVolNumPages{Pacific J. Math.}{31}{}{143--156}.
\PrintBackRefs{\CurrentBib}

\bibitem [\protect \citeauthoryear {%
Matsuoka%
}{%
Matsuoka%
}{%
{\protect \APACyear {1991}}%
}]{%
Mat91}
\APACinsertmetastar {%
Mat91}%
\begin{APACrefauthors}%
Matsuoka, S.%
\end{APACrefauthors}%
\unskip\
\newblock
\APACrefYearMonthDay{1991}{}{}.
\newblock
{\BBOQ}\APACrefatitle {Nonsingular algebraic curves in {${\bf R}{\rm
  P}^1\times{\bf R}{\rm P}^1$}} {Nonsingular algebraic curves in {${\bf R}{\rm
  P}^1\times{\bf R}{\rm P}^1$}}.{\BBCQ}
\newblock
\APACjournalVolNumPages{Trans. Amer. Math. Soc.}{324}{1}{87--107}.
\PrintBackRefs{\CurrentBib}

\bibitem [\protect \citeauthoryear {%
Nagami%
}{%
Nagami%
}{%
{\protect \APACyear {2000}}%
}]{%
Nag00}
\APACinsertmetastar {%
Nag00}%
\begin{APACrefauthors}%
Nagami, S.%
\end{APACrefauthors}%
\unskip\
\newblock
\APACrefYearMonthDay{2000}{}{}.
\newblock
{\BBOQ}\APACrefatitle {Existence of {S}pin structures on double branched
  covering spaces over four-manifolds} {Existence of {S}pin structures on
  double branched covering spaces over four-manifolds}.{\BBCQ}
\newblock
\APACjournalVolNumPages{Osaka J. Math.}{37}{2}{425--440}.
\PrintBackRefs{\CurrentBib}

\bibitem [\protect \citeauthoryear {%
Rokhlin%
}{%
Rokhlin%
}{%
{\protect \APACyear {1978}}%
}]{%
Rok78}
\APACinsertmetastar {%
Rok78}%
\begin{APACrefauthors}%
Rokhlin, V\BPBI A.%
\end{APACrefauthors}%
\unskip\
\newblock
\APACrefYearMonthDay{1978}{}{}.
\newblock
{\BBOQ}\APACrefatitle {Complex topological characteristics of real algebraic
  curves} {Complex topological characteristics of real algebraic
  curves}.{\BBCQ}
\newblock
\APACjournalVolNumPages{Uspekhi Mat. Nauk}{33}{5(203)}{77--89, 237}.
\PrintBackRefs{\CurrentBib}

\bibitem [\protect \citeauthoryear {%
Viro%
}{%
Viro%
}{%
{\protect \APACyear {1984}}%
}]{%
Vir84}
\APACinsertmetastar {%
Vir84}%
\begin{APACrefauthors}%
Viro, O\BPBI Y.%
\end{APACrefauthors}%
\unskip\
\newblock
\APACrefYearMonthDay{1984}{}{}.
\newblock
{\BBOQ}\APACrefatitle {Progress during the last five years in the topology of
  real algebraic varieties} {Progress during the last five years in the
  topology of real algebraic varieties}.{\BBCQ}
\newblock
\BIn{} \APACrefbtitle {Proceedings of the {I}nternational {C}ongress of
  {M}athematicians, {V}ol. 1, 2 ({W}arsaw, 1983)} {Proceedings of the
  {I}nternational {C}ongress of {M}athematicians, {V}ol. 1, 2 ({W}arsaw,
  1983)}\ (\BPGS\ 603--619).
\newblock
\APACaddressPublisher{}{PWN, Warsaw}.
\PrintBackRefs{\CurrentBib}

\bibitem [\protect \citeauthoryear {%
Viro%
\ \BBA {} Zvonilov%
}{%
Viro%
\ \BBA {} Zvonilov%
}{%
{\protect \APACyear {1992}}%
}]{%
VZ92}
\APACinsertmetastar {%
VZ92}%
\begin{APACrefauthors}%
Viro, O\BPBI Y.%
\BCBT {}\ \BBA {} Zvonilov, V\BPBI I.%
\end{APACrefauthors}%
\unskip\
\newblock
\APACrefYearMonthDay{1992}{}{}.
\newblock
{\BBOQ}\APACrefatitle {An inequality for the number of nonempty ovals of a
  curve of odd degree} {An inequality for the number of nonempty ovals of a
  curve of odd degree}.{\BBCQ}
\newblock
\APACjournalVolNumPages{Algebra i Analiz}{4}{3}{159--170}.
\PrintBackRefs{\CurrentBib}

\bibitem [\protect \citeauthoryear {%
\v{S}arafutdinov%
}{%
\v{S}arafutdinov%
}{%
{\protect \APACyear {1973}}%
}]{%
Sha73}
\APACinsertmetastar {%
Sha73}%
\begin{APACrefauthors}%
\v{S}arafutdinov, V\BPBI A.%
\end{APACrefauthors}%
\unskip\
\newblock
\APACrefYearMonthDay{1973}{}{}.
\newblock
{\BBOQ}\APACrefatitle {Relative {E}uler class and the {G}auss-{B}onnet theorem}
  {Relative {E}uler class and the {G}auss-{B}onnet theorem}.{\BBCQ}
\newblock
\APACjournalVolNumPages{Sibirsk. Mat. \v{Z}.}{14}{}{1321--1335, 1367}.
\PrintBackRefs{\CurrentBib}

\bibitem [\protect \citeauthoryear {%
Wilson%
}{%
Wilson%
}{%
{\protect \APACyear {1978}}%
}]{%
Wil78}
\APACinsertmetastar {%
Wil78}%
\begin{APACrefauthors}%
Wilson, G.%
\end{APACrefauthors}%
\unskip\
\newblock
\APACrefYearMonthDay{1978}{}{}.
\newblock
{\BBOQ}\APACrefatitle {Hilbert's sixteenth problem} {Hilbert's sixteenth
  problem}.{\BBCQ}
\newblock
\APACjournalVolNumPages{Topology}{17}{1}{53--73}.
\PrintBackRefs{\CurrentBib}

\bibitem [\protect \citeauthoryear {%
Yamada%
}{%
Yamada%
}{%
{\protect \APACyear {1995}}%
}]{%
Yam95}
\APACinsertmetastar {%
Yam95}%
\begin{APACrefauthors}%
Yamada, Y.%
\end{APACrefauthors}%
\unskip\
\newblock
\APACrefYearMonthDay{1995}{}{}.
\newblock
{\BBOQ}\APACrefatitle {An extension of {W}hitney's congruence} {An extension of
  {W}hitney's congruence}.{\BBCQ}
\newblock
\APACjournalVolNumPages{Osaka J. Math.}{32}{1}{185--192}.
\PrintBackRefs{\CurrentBib}

\bibitem [\protect \citeauthoryear {%
Zvonilov%
}{%
Zvonilov%
}{%
{\protect \APACyear {1979}}%
}]{%
Zvo79}
\APACinsertmetastar {%
Zvo79}%
\begin{APACrefauthors}%
Zvonilov, V\BPBI I.%
\end{APACrefauthors}%
\unskip\
\newblock
\APACrefYearMonthDay{1979}{}{}.
\newblock
{\BBOQ}\APACrefatitle {Strengthened {P}etrovskii and {A}rnol'd inequalities for
  curves of odd degree} {Strengthened {P}etrovskii and {A}rnol'd inequalities
  for curves of odd degree}.{\BBCQ}
\newblock
\APACjournalVolNumPages{Funktsional. Anal. i Prilozhen.}{13}{4}{31--39, 96}.
\PrintBackRefs{\CurrentBib}

\bibitem [\protect \citeauthoryear {%
Zvonilov%
}{%
Zvonilov%
}{%
{\protect \APACyear {2022}}%
}]{%
Zvo22}
\APACinsertmetastar {%
Zvo22}%
\begin{APACrefauthors}%
Zvonilov, V\BPBI I.%
\end{APACrefauthors}%
\unskip\
\newblock
\APACrefYearMonthDay{2022}{}{}.
\newblock
{\BBOQ}\APACrefatitle {Viro-{Z}vonilov inequalities for flexible curves on an
  almost complex four-dimensional manifold} {Viro-{Z}vonilov inequalities for
  flexible curves on an almost complex four-dimensional manifold}.{\BBCQ}
\newblock
\APACjournalVolNumPages{Lobachevskii J. Math.}{43}{3}{720--727}.
\PrintBackRefs{\CurrentBib}

\end{thebibliography}
\bibliographystyle{apacite}

\end{document}